\RequirePackage{fix-cm}
\documentclass[smallextended]{svjour3}       
\smartqed  
\usepackage{graphicx}

\usepackage{amsfonts,amssymb,amsmath,amscd}
\usepackage[mathscr]{eucal}
\usepackage{mathrsfs}
\usepackage{color}
\usepackage{enumerate}  

\newtheorem{theo}{Theorem}
\newtheorem{cor}{Corollary}
\newtheorem{lem}{Lemma}

\newtheorem{df}{Definition}

\newtheorem{rem}{Remark}

\def \N{{\mathbb N}}
\def \Z{{\mathbb Z}}
\def \Q{{\mathbb Q}}

\def \F{{\mathbb F}}
\def \P{{\mathbb P}}

\def \ss{\underset{ R}{\sim}}
\def \s{\underset{R}{\sim}^*}
\def \tt{\overset{\sim}{\longrightarrow}}

\makeatletter

\@addtoreset{equation}{section}
\renewcommand{\p@enumii}{}
\makeatother

\begin{document}

\title{On the Laxton Group
}


\author{Miho Aoki        \and
        Masanari Kida
}


\institute{ M. Aoki \at
 Department of Mathematics,
Interdisciplinary Faculty of Science and Engineering,
Shimane University,
Matsue, Shimane, 690-8504, Japan \\
              \email{aoki@riko.shimane-u.ac.jp}           
           \and
 M. Kida  \at
 Department of Mathematics, 
Faculty of Science Division I, 
Tokyo University of Science, 
1-3 Kagurazaka Shinjuku, Tokyo, 162-8601, Japan \\
               \email{} 
}

\date{Received: date / Accepted: date}

\maketitle

\begin{abstract}
In 1969, Laxton defined a multiplicative group structure on the set of
 rational sequences
 satisfying a fixed linear recurrence of degree two. 
He also defined some natural subgroups of the group, and determined the
structures of their quotient groups.
Nothing has been known about the structure of Laxton's whole group
and its interpretation.
In this paper, we redefine his group
 in a natural way and determine the structure of the whole group,
 which  clarifies Laxton's results on the quotient groups.
This definition makes us possible to use
 the group to show various properties of such sequences.
\end{abstract}
\section{The Laxton group}\label{sec:LG}
Let $P$ and $Q$ are 
integers with $Q\ne 0$.

We consider linear recurrence sequences associated to the characteristic
 polynomial $f(t):=t^2-Pt+Q$. Namely, they are determined by
 $w_{n+2}-Pw_{n+1}+Qw_n=0$ with the  first rational numbers $w_0$ and $w_1$.
 Let $\theta_1$ and $\theta_2$ be the roots of $f(t)$.
 We assume 
 \[
D:=\mathrm{disc}(f)=P^2-4Q=(\theta_1-\theta_2)^2\ne 0.
\]

 We define an equivalence relation $\sim^*$ on the set of the linear recurrence sequences.
 For ${\bf w}=(w_n),\ {\bf v}=(v_n)$, we write ${\bf w} \sim^* {\bf v}$ if there 
exists $\lambda \in \Q^{\times}$
and $\nu \in \Z$ 
such that $w_n=\lambda v_{n+\nu}$ for any $n\in \Z$.
Laxton \cite{L1} considered the following quotient set using this relation:
\begin{align*}
G^*(f) :=& \{ (w_n)_{n\in \Z} \ |\ w_0,w_1 \in \Q \ \text{with}\ \Lambda (w_1,w_0)\ne 0,\ 
w_{n+2}-Pw_{n+1}+Qw_n=0 \\ 
& \text{for any}\ n\in \Z, 
\  \text{and there exists $\nu\in \Z$ such that $w_k \in \Z$ for any $k\geq \nu$ } \}/\sim^*,
\end{align*}
where $\Lambda (w_1,w_0):=w_1^2-Pw_0w_1+Qw_0^2$.
%
Laxton
introduced a product on $G^*(f)$ as follows. For classes ${\mathcal W}$ and ${\mathcal V} \in G^*(f)$,
let  $(w_n)$ and $(v_n)$ are representatives of the classes ${\mathcal W}$ and ${\mathcal V}$, respectively.
The product ${\mathcal W} \times {\mathcal V}$ of ${\mathcal W}$ and ${\mathcal V}$ is, then, defined by the class of $(u_n)$ with $u_n=(AC \theta_1^n-BD\theta_2^n)/
(\theta_1-\theta_2)$, where $A=w_1-w_0\theta_2,\ B=w_1-w_0\theta_1, \ C=v_1-v_0 \theta_2$ and $D=v_1-v_0 \theta_1$.
We call $G^*(f)$ the Laxton group.
Let $p$ be a prime number with $p\nmid Q$.
For ${\bf w}=(w_n)$, if $w_n\in p\Z_{(p)}$ for some $n$, we say
$p$ is a divisor of ${\bf w}$ and write $p|{\bf w}$,
where $\Z_{(p)}$ is the localization of the integers at $p$.
Laxton defined the following subgroups of $G^*(f)$:
\begin{align*} 
G^*(f,p)_L := & \{ {\mathscr W} \in G^*(f) \ |\  p |{\bf w} \ \text{for\ all}\  {\bf w}\in {\mathscr W} \},\\
K^*(f,p)_L := & \{ {\mathscr W} \in G^*(f) \ |\  \text{ there exists ${\bf w} \in {\mathscr W} $
for\ which 
$w_0,w_1 \in \mathbb Z$ and $p\nmid \Lambda (w_1,w_0)$}\}, 
\\
H^*(f,p)_L := & \{ {\mathscr W} \in G^*(f) \ |\  {\mathscr W}^n \in G^*(f,p)_L \ \text{for\ some}\ n \in \Z \}.
\end{align*}
Laxton proved the following theorem on the quotient groups.
\begin{theo}[\protect{Laxton \cite[Theorem~3.7]{L1}}]\label{theo:Laxton}

 Let $p$ be a prime number with $p\nmid Q$, and $r(p)$ be the rank of the Lucas sequence ${\mathcal F}$
$($see Definitions~\ref{df:Lucas} and ~\ref{df:rank}$)$.
\begin{enumerate}[$(1)$]
\item  Assume $p\nmid D$. If $\mathbb Q(\theta_1)\ne \mathbb Q$ and $p$ is inert in $\mathbb Q(\theta_1)$,
 then $G^*(f)=H^*(f,p)_L=K^*(f,p)_L$ and
$G^*(f)/G^*(f,p)_L$ is a cyclic group of order $(p+1)/r(p)$.
\label{enum:Laxton1}
\item Assume $p\nmid D$. If $\Q (\theta_1)=\Q$,  or  $\Q(\theta_1)\ne \Q$ and
$p$ splits in $\Q(\theta_1)$, then $G^*(f)/H^*(f,p)_L$ is an infinite cyclic group,
and $H^*(f,p)_L=K^*(f,p)_L$ and $H^*(f,p)_L/G^*(f,p)_L$ is a cyclic group of order $(p-1)/r(p)$.
\label{enum:Laxton2}
\item If  $p|D$ and $p^2\nmid D$, then $G^*(f)=H^*(f,p)_L$ and $K^*(f,p)_L=G^*(f,p)_L$.
Furthermore, if $p\ne 2$, then $G^*(f)/G^*(f,p)_L$ is a cyclic group of order two.
\label{enum:Laxton3}
\end{enumerate}
\end{theo}

Laxton made no assumption in $p^2 \nmid D$ of Theorem~\ref{theo:Laxton} (\ref{enum:Laxton3}).
Suwa \cite{S} recently pointed out that this assumption is necessary and
gave  counterexamples that did not hold in the case $p^2  \mid D$ of Theorem~\ref{theo:Laxton} (\ref{enum:Laxton3}),
and gave the correct formulation above and the proof for the case 
(Corollary~\ref{cor:main}(\ref{enum:cormain4})).
He also gives in his paper an interpretation of
linear recurrence sequences of degree two from a viewpoint of the theory of group schemes.
%

Althogh Laxton studied structures of the quotient groups of $G^*(f)$, he did not 
study the group $G^*(f)$ itself. The aims of this paper are to redefine Laxton's group in
a natural way (Theorems~\ref{theo:main1} and \ref{theo:Vf}) and to give  structure theorems of the group itself
and the subgroups.
One of our main results is the following theorem.
\begin{theo}[\protect{Theorems~\ref{theo:G*}, \ref{theo:main2} and \ref{theo:KK*}}]\label{theo:main}
Notations being as above.
 Put $D=p^sD_0$ with $s\geq 0,\ p\nmid D_0$. 
\begin{enumerate}[$(1)$]
\item  If $f(t)$ is irreducible over $\Q$, then we have
\[
G^*(f)\tt \Q(\theta_1)^{\times}/\Q^{\times}\langle \theta_1 \rangle.
\]

If $f(t)$ is reducible over $\Q$, then we have 
\[
G^*(f)\tt \Q^{\times} /
\langle \theta_1 \theta_2^{-1} \rangle.
\]
\label{enum:main1}
\item There exists the following exact sequence
\[
1 \longrightarrow G^*(f,p)_L \longrightarrow K^*(f,p)_L \underset{ \mathrm{red}_p}{\longrightarrow}  G^*_{\F_p}(f) \longrightarrow 1
\]
where $G^*_{\F_p}(f)$ is the group of equivalence classes of linear recurrence sequences over the finite field $\F_p$.
\label{enum:main2}
\item If $f(t)$ is irreducible over $\Q$, then we have
\[
K^*(f,p)_L \tt
\Z_{(p)}[\theta_1]^{\times}/\Z_{(p)}^{\times}\langle \theta_1 \rangle.
\]
If $f(t)$ is reducible over $\Q$, then we have
\[
K^*(f,p)_L \tt \begin{cases}
\Z_{(p)}^{\times}/\langle \theta_1 \theta_2^{-1} \rangle  & \text{if}\ p\nmid D,\\
(1+p^{\frac{s}{2}} \Z_{(p)} )/\langle \theta_1 \theta_2^{-1} \rangle & \text{if}\ p|D.
\end{cases} 
\]
\label{enum:main3}
\end{enumerate}
\end{theo}

The content of this paper is as follows.
In \S~\ref{sec:LR}, we begin with redefining the group law
on the set of linear recurrence sequences,
which gives a natural interpretation of Laxton's product.
In \S~\ref{sec:GS}, we determine the group structures of the group of the set of linear recurrence
sequences
according to the irreducibility of the  characteristic polynomial $f(t)$.
In \S~\ref{sec:EC}, we introduce two relations on the group of linear recurrence sequences,
and determine the group structures of the quotient groups by these
relations. In particular, we can determine that of the Laxton group $G^*(f)$.
In \S~\ref{sec:SubG}, we redefine the subgroups $G^*(f,p)_L,\ K^*(f,p)_L$ and
$H^*(f,p)_L$  for a fixed prime number $p$.
From our new definitions, we see that $G^*(f,p)_L$ is the kernel of the reduction map from $K^*(f,p)_L$
to the group $G^*_{\mathbb F_p}(f)$ of equivalence classes of linear recurrence
sequences of the finite filed $\F_p$. Furthermore, we study the relations between these subgroups and the
rank $r(p)$ of the Lucas sequence, which is a classical invariant concerning Artin's conjecture on primitive roots.
In \S~\ref{sec:K}, \ref{sec:G/Kp} and \ref{sec:Kp/Gp}, we determine the group structures of $K^*(f,p)_L,\
G^*(f)/K^*(f,p)_L$ and $K^*(f,p)/G^*(f,p)$, respectively by using $p$-adic logarithm functions.
The results in these sections yield Laxton's (Theorems~\ref{theo:Laxton}) and Suwa's theorems.
Suwa first gave a proof in the case $p^2 |D$ from a viewpoint of the theory of group schemes.

%
\section{Group laws of Linear recurrence sequences}\label{sec:LR}

Let $R$ be an integral domain. In order to discuss in the general situation, we consider the
sequences
$(w_n)$ of $R$ determined by 
\[
w_{n+2}-Pw_{n+1}+Qw_n=0,
\]
for fixed elements $P,Q \in R$ with $Q\in R^{\times}$. Let $\theta_1$ and $\theta_2$ be
the roots of the characteristic polynomial 
\[
f(t):=t^2-Pt+Q \ \in R[t]
\]
in an algebraic closure $\overline{k}_R$ of the quotinet field $k_R$ of $R$. Put
\[
d:=\mathrm{disc}(f)=P^2-4Q \ \in R.
\]
Define
\[
\mathscr{S}(f,R) := \{ (w_n)_{n\in \Z} \ \vline \ w_0,w_1 \in R,\ w_{n+2}-Pw_{n+1}+Qw_n=0\
\text{ for any }\ n\in \Z \}.
\]
By the assumption  $Q\in R^{\times}$,  any  sequence $(w_n) \in \mathscr{S}(f,R)$ satisfies  $w_n \in R$ for any $n\in \Z$.
All sequences in $\mathscr{S}(f,R)$ have the characteristic polynomial $f(t)$, and the set $\mathscr{S}(f,R)$
has a natural $R$-module strcture and is isomorphic to a free $R$-module of rank $2$ :
\[
V_f(R) := \left\{ \begin{pmatrix}
a_1\\
a_0\\
\end{pmatrix} \ \vline \ a_1,a_0\in R \right\}
\]
with an isomorphism given by 
\begin{equation}\label{eq:LR1}
\phi_R : \ \mathscr{S}(f,R) \tt V_f(R), \quad (w_n) \mapsto \begin{pmatrix}
w_1\\
w_0\\
\end{pmatrix}.
\end{equation}
Furthermore, we define an isomorphism of $R$-modules $\varphi_R$ by
\begin{equation}\label{eq:LR2}
\varphi_R : \ V_f(R) \tt {\mathcal O}_R, \quad \begin{pmatrix}
a_1\\
a_0\\
\end{pmatrix}  \mapsto a_1-a_0t \bmod{f(t)},
\end{equation}
where ${\mathcal O}_R:=R[t]/(f(t))$ is the quotient ring of $R[t]$.
We define ring structures on $V_f(R)$ and $\mathscr{S}(f,R) $
via the maps $\phi_R$ and $\varphi_R$ induced from that of ${\mathcal O}_R$. 
Put
\[
{\mathcal B}:=\begin{pmatrix}
P & -Q\\
1 & 0\\
\end{pmatrix}.
\]
Since det$({\mathcal B})=Q\in R^{\times}$, we have ${\mathcal B}\in \mathrm{GL}_2(R)$.
There is a natural left action of the group $\langle {\mathcal B} \rangle $ on 
$V_f(R)$ because the matrix ${\mathcal B}$ satisfies 
\begin{equation}\label{eq:LR3}
\begin{pmatrix}
w_{n+1} \\
w_n \\
\end{pmatrix}={\mathcal B}^n \begin{pmatrix}
w_1\\
w_0\\
\end{pmatrix},
\end{equation}
for any $n \in \mathbb Z$. We define left actions of $\langle {\mathcal B} \rangle$ on
$\mathscr{S}(f,R)$ and ${\mathcal O}_R$ via the maps $\phi_R$ and $\varphi_R$,
respectively.
From 
\begin{equation}\label{eq:LR4}
\varphi_R \left( {\mathcal B} \begin{pmatrix}
a_1\\
a_0\\
\end{pmatrix} \right) =(Pa_1-Qa_0)-a_1t =(P-t)(a_1-a_0t),
\end{equation}
for any $\begin{pmatrix}
a_1\\
a_0\\
\end{pmatrix} \in V_f(R)$, the action of ${\mathcal B}$ on ${\mathcal O}_R$ is given by
the multiplication by $P-t$. In particular, we get
\begin{equation}\label{eq:LR5}
{\mathcal B} \begin{pmatrix}
a_1\\
a_0\\
\end{pmatrix} *\begin{pmatrix}
b_1\\
b_0\\
\end{pmatrix} =\begin{pmatrix} 
a_1\\
a_0\\
\end{pmatrix} *{\mathcal B} \begin{pmatrix}
b_1\\
b_0\\
\end{pmatrix} ={\mathcal B} \left\{ 
\begin{pmatrix} 
a_1\\
a_0\\
\end{pmatrix} *\begin{pmatrix}
b_1\\
b_0\\
\end{pmatrix} \right \},
\end{equation}
for any $\begin{pmatrix}
a_1\\
a_0\\
\end{pmatrix}, \begin{pmatrix}
b_1\\
b_0\\
\end{pmatrix} \in V_f(R)$.
Furthermore, we get
\begin{align}
 & {\mathcal B}^n \begin{pmatrix}
a_1\\
a_0\\
\end{pmatrix} =\begin{pmatrix}
P\\
1\\
\end{pmatrix}^n *\begin{pmatrix}
a_1\\
a_0\\
\end{pmatrix}, \label{eq:LR6}
\end{align}
for any $n\in \mathbb Z$.
For a class of $a_1-a_0t$ in ${\mathcal O}_R$,
we have
\begin{equation}\label{eq:LR7}
(a_1-a_0t) (1\ \ -t) =(1\ \ -t)\begin{pmatrix}
a_1 & -a_0 Q\\
a_0 & a_1-a_0P \\
\end{pmatrix}.
\end{equation}
%
 Define the norm of a class of $a_1-a_0t \in {\mathcal O}_R$ by 
\begin{align}
N(a_1-a_0t) & :=\mathrm{det} \begin{pmatrix}
a_1 & -a_0 Q\\
a_0 &a_1-a_0P \\
\end{pmatrix} \quad (\in R) \label{eq:LR8} \\
&=a_1^2-Pa_0a_1+Qa_0^2 \notag \\
&=(a_1-\theta_1 a_0)(a_1-\theta_2 a_0) \notag 
\end{align}
\begin{rem}\label{rem:lambda}
 If $f(t)$ is irreducible over the quotient field $k_R$ of $R$, then $N(a_1-a_0t)\ne 0$ if and only if
$a_0\ne 0$ or $a_1\ne 0$.
\end{rem}
By the definition of the norm and (\ref{eq:LR7}), the norm is multiplicative, namely we have
\[
N((a_1-a_0t)(b_1-b_0t))=N(a_1-a_0t)N(b_1-b_0t),
\]
for any classes of $a_1-a_0t,\ b_1-b_0 t$ of ${\mathcal O}_R$.
From the fact, we can see that a class of $a_1-a_0t$ is invertible in ${\mathcal O}_R$ if and only if
$N(a_1-a_0t)\in R^{\times}$.
In particular, since 
\[
(P-t)(1\ \ -t)=(1\ \ -t){\mathcal B},
\]
we have $N(P-t)=\mathrm{det}({\mathcal B})=Q \in R^{\times}$,
and hence $P-t$ is invertible in ${\mathcal O}_R$.
Since the norm is multiplicative, (\ref{eq:LR3}) and (\ref{eq:LR4}) yield
\begin{equation}\label{eq:LR9}
N(w_{n+1}-w_nt) =N(P-t)^n N(w_1-w_0t)  
=Q^n N(w_1-w_0t), 
\end{equation}
for any integer $n$ and any sequaence $(w_n) \in \mathscr{S}(f,R)$.
We can endow the inverse image of ${\mathcal O}_R^{\times}$ by $\varphi_R$:
\[
V_f(R)^{\times} :=\varphi_R^{-1} ({\mathcal O}_R^{\times} )=\left\{ \begin{pmatrix}
a_1\\
a_0\\
\end{pmatrix} \in V_f(R) \ \vline \ \Lambda (a_1,a_0)\in R^{\times} \right\}
\]
where $\Lambda(a_1,a_0):=N(a_1-a_0t)$, and its inverse image by $\phi_R$ :
\[
\mathscr{S}(f,R)^{\times} := \phi_R^{-1} (V_f(R)^{\times} )=\left\{ (w_n) \in \mathscr{S}(f,R)
\ \vline  \ \Lambda (w_1,w_0) \in R^{\times} \right\}
\]
with the structure of a multiplicative group induced from ${\mathcal O}_R^{\times}$.
If $\begin{pmatrix}
a_1\\
a_0\\
\end{pmatrix}$ and $\begin{pmatrix}
b_1\\
b_0\\
\end{pmatrix}$ are  elements of $V_f(R)^{\times}$, then the corresponding product in ${\mathcal O}_R^{\times}$
is
\[
(a_1-a_0 t)(b_1-b_0t)\equiv (a_1b_1-a_0b_0 Q)-(a_0b_1+a_1b_0-Pa_0b_0)t \pmod{f(t)}.
\]
Thus the multiplication in $V_f(R)^{\times}$ is given by
\begin{equation}\label{eq:LR10}
\begin{pmatrix}
a_1\\
a_0\\
\end{pmatrix} * \begin{pmatrix}
b_1\\
b_0\\
\end{pmatrix} =\begin{pmatrix}
a_1b_1-Qa_0b_0\\
a_0b_1+a_1b_0-Pa_0b_0 \\
\end{pmatrix}.
\end{equation}
The identity element is $\begin{pmatrix}
1\\
0\\
\end{pmatrix}$ and $\begin{pmatrix}
a_1 \\
a_0\\
\end{pmatrix}^{-1} =\Lambda (a_1,a_0)^{-1} \begin{pmatrix}
a_1-Pa_0\\
-a_0\\
\end{pmatrix}$.
%
If $\begin{pmatrix}
a_1\\
a_0\\
\end{pmatrix} \in V_f(R)^{\times}$, then we have $ \varphi_R \left( {\mathcal B} \begin{pmatrix}
a_1\\
a_0\\
\end{pmatrix} \right) \in {\mathcal O}_R^{\times}$,
from (\ref{eq:LR4}) and $P-t \in {\mathcal O}_R^{\times}$,
and hence $ {\mathcal B} \begin{pmatrix}
a_1\\
a_0\\
\end{pmatrix} \in V_f(R)^{\times}$. 
Therefore, the multiplicative groups $V_f(R)^{\times}, {\mathcal O}_R^{\times}$ and
$\mathscr{S}(f,R)^{\times}$ have the left actions of $\langle {\mathcal B} \rangle$.

The roots $\theta_1$ and $\theta_2$ of $f(t)$ are the  eigenvalues of ${\mathcal B}$ and the
corresponding eigenspaces are
$\langle \begin{pmatrix}
\theta_1 \\
1\\
\end{pmatrix}
\rangle$ and $\langle \begin{pmatrix}
\theta_2 \\
1\\
\end{pmatrix}
\rangle$ respectively.
Hence,
if $d\ne 0$, then
${\mathcal B}$ is diagonalized by ${\mathcal P}:= \begin{pmatrix}
\theta_1 & \theta_2 \\
1 & 1\\
\end{pmatrix} \in \mathrm{GL}_2(\overline{k}_R)$ as
\[
{\mathcal P}^{-1} {\mathcal B} {\mathcal P} =\begin{pmatrix}
\theta_1 & 0\\
0 & \theta_2 \\
\end{pmatrix}.
\]
%
Let ${\bf w}=(w_n)_{n\in \Z} \in \mathscr{S}(f,R)$. 
 The general term is given by
 \[
 w_n=\frac{A\theta_1^n-B\theta_2^n}{\theta_1-\theta_2},
 \]
 where $A=w_1-w_0\theta_2$ and $B=w_1-w_0\theta_1$. 
We have
\[
AB=\Lambda(w_1,w_0) \quad \in R.
\]
\begin{df}\label{df:Lucas}
 We call the sequence ${\mathcal F}=({\mathcal F}_n) \in \mathscr{S}(f,R) $ with ${\mathcal F}_0=0$ and ${\mathcal F}_1=1$
{\it the Lucas sequence}, and the sequence  ${\mathcal L}=({\mathcal L}_n) \in \mathscr{S}(f,R)$ with 
${\mathcal L}_0=2$ and ${\mathcal L}_1=P$
{\it the companion Lucas sequence} after Lucas, who first introduced them in  \cite{Lu}.
Their general terms are $\displaystyle{ {\mathcal F}_n=\frac{\theta_1^n -\theta_2^n}{\theta_1-\theta_2}}$ and 
${\mathcal L}_n=\theta_1^n+\theta_2^n$.
\end{df}

Next, we recall a product on the sets $\mathscr{S}(f,R)^{\times}$ introduced by Laxton \cite{L1} (see also \cite{B}).
\begin{df}\label{df:LG-law}
 Let ${\bf w}=(w_n), {\bf v}=(v_n)\in \mathscr{S}(f,R)^{\times}$. Write
\[
w_n=\frac{A\theta_1^n-B\theta_2^n}{\theta_1-\theta_2}, \quad v_n=\frac{C\theta_1^n-D\theta_2^n}{\theta_1-\theta_2},
\]
where $A=w_1-w_0\theta_2, B=w_1-w_0\theta_1, C=v_1-v_0\theta_2$ and $D=v_1-v_0\theta_1$.
Laxton defined the product ${\bf w}\times {\bf v}={\bf u}=(u_n)$ by
\[
u_n=\frac{AC\theta_1^n-BD\theta_2^n}{\theta_1-\theta_2},
\]
for any $n\in \Z$. In particular, $u_0$ and $u_1$ are given by
$u_0=w_0v_1+w_1v_0-Pv_0w_0,\ u_1=w_1v_1-Qv_0w_0$.
We get ${\bf u}\in \mathscr{S}(f,R)^{\times}$ since
$\Lambda(u_1,u_0)=ABCD=\Lambda(w_1,w_0)\Lambda(v_1,v_0)\in R^{\times}$.
The associativity is trivial, and the identity is the Lucas sequence ${\mathcal F}=({\mathcal F}_n)$.
The inverse element of 
\[
 {\bf w}=(w_n), \quad w_n=\frac{A\theta_1^n-B\theta_2^n}{\theta_1-\theta_2}
\]
 is given by 
 \[
{\bf w}^{-1}=(u_n), \quad
u_n=\Lambda (w_1,w_0)^{-1}\frac{B\theta_1^n-A\theta_2^n}{\theta_1-\theta_2}.
\]
\end{df}
The multiplicative group structure on $\mathscr{S}(f,R)^{\times}$ defined by Laxton
coincides with one induced from ${\mathcal O}_R^{\times}=(R[t]/(f(t)))^{\times}$ via maps 
$\phi_R$ and $\varphi_R$ in (\ref{eq:LR1}) and (\ref{eq:LR2}), respectively.
We get the following theorem.
\begin{theo}\label{theo:main1} 
Let $R$ be an integral domain. The group strcture of $\mathscr{S}(f,R)^{\times}$ defined by
Laxton coincides with one induced from ${\mathcal O}_R^{\times} =(R[t]/(f(t)))^{\times}$ via the maps 
$\phi_R$ and $\varphi_R$:
\begin{align}
&\mathscr{S}(f,R)^{\times} \overset{\sim}{ \underset{\phi_R}{\longrightarrow}} V_f(R)^{\times} 
\overset{\sim}{ \underset{\varphi_R}{\longrightarrow}} {\mathcal O}_R^{\times} =(R[t]/(f(t)))^{\times}, \notag
\\
& {\bf w}=(w_n) \  \mapsto \ \begin{pmatrix}
w_1\\
w_0\\
\end{pmatrix} \ \mapsto \ w_1-w_0t. \notag
\end{align}
Furthermore, these group isomorphisms are compatible with the following action of $\langle {\mathcal B} \rangle$:
\begin{enumerate}[$(i)$]
\item  For ${\bf w}=(w_n) \in \mathscr{S}(f,R)^{\times}$ and $\nu \in \Z$,
\[
{\mathcal B}^{\nu}. {\bf w} ={\bf v} =(v_n),
\]
where $v_n=w_{n+\nu}$ for any $n\in \Z$.
\label{enum:1maini}
\item  For $\begin{pmatrix}
w_1\\
w_0\\
\end{pmatrix} \in V_f(R)^{\times}$ and $\nu \in \Z$,
\[
{\mathcal B}^{\nu}. \begin{pmatrix}
w_1\\
w_0\\
\end{pmatrix}={\mathcal B}^{\nu} \begin{pmatrix}
w_1\\
w_0\\
\end{pmatrix}
\]
(the right-hand side is the ordinary matrix product).
\label{enum:1mainii}
\item For $w_1-w_0t \in {\mathcal O}_R^{\times}$ and $\nu \in \Z$,
\[
{\mathcal B}^{\nu}. (w_1-w_0t) =(P-t)^{\nu} (w_1-w_0t).
\]
\label{enum:1mainiii}
\end{enumerate}
\end{theo}
\section{Group structures of $\mathscr{S}(f,R)^{\times}$}\label{sec:GS}
In this section, we assume that $R$ is a unique factorization domain, and
 study the group structures of
\begin{align}
&\mathscr{S}(f,R)^{\times} \overset{\sim}{ \underset{\phi_R}{\longrightarrow}} V_f(R)^{\times} 
\overset{\sim}{ \underset{\varphi_R}{\longrightarrow}} {\mathcal O}_R^{\times} =(R[t]/(f(t)))^{\times}, \label{eq:GS1}
\\
& {\bf w}=(w_n) \  \mapsto \ \begin{pmatrix}
w_1\\
w_0\\
\end{pmatrix} \ \mapsto \ w_1-w_0t \notag
\end{align}
according to the irreducibility of the polynomial $f(t)$.
\begin{theo}\label{theo:Vf}
\begin{enumerate}[$(1)$]
\item  If $f(t)$ is irreducible over $R$, then we have an isomorphism of $R$-algebras
\[
\psi_R:\ V_f(R) \tt R[\theta_1], \quad \begin{pmatrix}
a_1\\
a_0\\
\end{pmatrix} \mapsto a_1-a_0\theta_1.
\]
This yields a group isomorphism
\[
\mathscr{S}(f,R)^{\times} \simeq V_f(R)^{\times} \simeq R[\theta_1 ]^{\times}.
\]
\label{enum:Vf1}
\item  Assume that $f(t)$ is reducible over $R$, and hence $d=\mathrm{disc}(f)\in R^2$.
\label{enum:Vf2}
\begin{enumerate}[$(i)$]
\item The case where $f(t)$ has no double root in $R$: Let $H_R$ be an $R$-subalgebra of $R\times R$ defined by
\[
H_R:= \{(x,y)\in R\times R \ \mid \ x\equiv y \pmod{ \sqrt{d} } \},
\]
$($if $d\in R^{\times}$, then $H_R=R\times R$ $)$.
We have an isomorphism of $R$-algebras
\[
\psi_R:\ V_f(R) \tt H_R, \quad \begin{pmatrix}
a_1\\
a_0\\
\end{pmatrix} \mapsto (a_1-a_0\theta_1, a_1-a_0\theta_2).
\]
This yields a group isomorphism
\[
\mathscr{S}(f,R)^{\times} \simeq V_f(R)^{\times} \simeq H_R^{\times}=\{ (x,y) \in R^{\times} \times R^{\times} \ \mid \ x\equiv y \pmod{\sqrt{d}} \}.
\] 
\label{enum:Vf2i}
\item The case where $f(t) $ has a double root $\theta$ in $R$: We have an isomorphism of $R$-algebras
\[
\psi_R :\ V_f(R) \tt R[\varepsilon]/(\varepsilon^2), \quad \begin{pmatrix}
a_1\\
a_0\\
\end{pmatrix} \mapsto a_1-a_0(\varepsilon+\theta) \bmod{\varepsilon^2}.
\]
 This yields a group isomorphism
 \[
 \mathscr{S}(f,R)^{\times} \simeq V_f(R)^{\times} \simeq (R[\varepsilon]/(\varepsilon^2))^{\times}=\{ x+y\varepsilon  \bmod{\varepsilon^2} \mid  x\in R^{\times},\ y\in R\}.
 \]
 \label{enum:Vf2ii}
\end{enumerate}
\end{enumerate}
\end{theo}

\begin{proof}
The assertion (\ref{enum:Vf1}) follows from (\ref{eq:LR2}) and the isomorphism
\[
{\mathcal O}_R=R[t]/(f(t)) \tt R[\theta_1], \quad g(t) \bmod{f(t)} \mapsto g(\theta_1).
\]
Next, we show the assertion (\ref{enum:Vf2}).

In the case (\ref{enum:Vf2i}), we have $\theta_1\ne \theta_2$ in $R$. Consider the following homomorphism of $R$-algebras
\begin{align}
& {\mathcal O}_R=R[t]/(f(t)) \longrightarrow  R[t]/(t-\theta_1) \times R[t]/(t-\theta_2)   \tt  R\times R, \label{eq:GS2}\\
& \qquad g(t) \bmod{f(t)} \ \mapsto   (g(t)  \bmod{t-\theta_1}, g(t) \bmod{t-\theta_2})   \mapsto 
\ (g(\theta_1),g(\theta_2)).\notag
\end{align}
The map is injective, and the image is in the set $H_R$ since $\theta_1 \equiv \theta_2 \pmod{\sqrt{d}}$.
Conversely, we have
\[
g(t) =\frac{1}{\theta_2-\theta_1} \{ y(t-\theta_1)-x(t-\theta_2) \} \quad \in R[t]
\]
for any $(x,y) \in H_R $, and $g(t) \bmod{f(t)}$ maps to $(x,y)$ by the map (\ref{eq:GS2}).
we get the isomorphism
\[
R[t]/(f(t)) \simeq H_R,
\]
and the assertion follows from the isomorphism and (\ref{eq:LR2}).

In the case (\ref{enum:Vf2ii}), we have $\theta_1=\theta_2 $ in $R$. The assertion follow from (\ref{eq:LR2}) and 
the following isomorphism of $R$-algebras
\[
{\mathcal O}_R=R[t]/(f(t)) \tt R+R\varepsilon, \quad a_1-a_0t \bmod{f(t)} \mapsto a_1-a_0 (\varepsilon+\theta).
\]
\qed
\end{proof}
\section{Equivalence classes}\label{sec:EC}

Let $R$ be a unique factorization domain, and assume $Q\in R^{\times}$.
In this section, we introduce two relations $\ss$ and $\s$ on the set $\mathscr{S}(f,R) $,
and consider the quotient sets of $\mathscr{S}(f,R)^{\times}$ by these  relations. 
The group structure of $\mathscr{S}(f,R)^{\times}$ defined in \S~\ref{sec:LR}  naturally
induce the group structures on the
quotient sets.
Note that we have $w_n \in R$ for any $n\in \Z$ by the assumption $Q\in R^{\times}$.
\begin{df}\label{df:eq}
Let ${\bf w}=(w_n),\ {\bf v}=(v_n) \in \mathscr{S}(f,R) $.
\begin{enumerate}[$(1)$]
\item We define ${\bf w} \ss {\bf v}$ if  there exists $\lambda \in R^{\times}$ such that
$w_n=\lambda v_n$ for any $n\in \Z$.
\label{enum:dfeq1}
\item We define ${\bf w} \s {\bf v}$ if there exists $\lambda \in R^{\times}$ and
$\nu \in \Z$ such that $w_n =\lambda v_{n+\nu} $ for any $n\in \Z$.
\label{enum:dfeq2}
\end{enumerate}
\end{df}
These relations are equivalence relations on the set $\mathscr{S}(f,R) $.
By the isomorphism (\ref{eq:LR1}), we can introduce the corresponding relations $\ss$ and $\s$ on the set
$V_f(R)$.
Let $\begin{pmatrix}
a_1\\
a_0\\
\end{pmatrix}, \begin{pmatrix}
b_1\\
b_0\\
\end{pmatrix} \in V_f(R)$.
\begin{itemize}
\item[(1)] We have $\begin{pmatrix}
a_1\\
a_0\\
\end{pmatrix} \ss \begin{pmatrix}
b_1\\
b_0\\
\end{pmatrix}$ if there exists $\lambda \in R^{\times}$ such that
$\begin{pmatrix}
a_1\\
a_0\\
\end{pmatrix}=\lambda \begin{pmatrix}
b_1\\
b_0\\
\end{pmatrix}$.
\item[(2)] We have $\begin{pmatrix}
a_1\\
a_0\\
\end{pmatrix} \s \begin{pmatrix}
b_1\\
b_0\\
\end{pmatrix}$ if there exist $\lambda \in R^{\times}$ and $\nu \in \Z$ such that
$\begin{pmatrix}
a_1\\
a_0\\
\end{pmatrix}=\lambda {\mathcal B}^{\nu} \begin{pmatrix}
b_1\\
b_0\\
\end{pmatrix}$.
\end{itemize}
If  $\begin{pmatrix}
a_1\\
a_0\\
\end{pmatrix} \s \begin{pmatrix}
b_1\\
b_0\\
\end{pmatrix}$, then  there exist $\lambda \in R^{\times}$ and $\nu \in \Z$ such that
$\begin{pmatrix}
a_1\\
a_0\\
\end{pmatrix}=\lambda {\mathcal B}^{\nu} \begin{pmatrix}
b_1\\
b_0\\
\end{pmatrix}$, hence we have from (\ref{eq:LR4})
\begin{align*}
\Lambda (a_1, a_0) &=N(a_1-a_0t)\\
&=\lambda^2 N(P-t)^{\nu} N(b_1-b_0t)\\
&=\lambda^2 Q^{\nu} \Lambda (b_1,b_0).
\end{align*}
We conclude that $\begin{pmatrix}
a_1\\
a_0\\
\end{pmatrix} \in V_f(R)^{\times}$ if and only if $\begin{pmatrix}
b_1\\
b_0\\
\end{pmatrix} \in V_f(R)^{\times}$.
\begin{df}\label{df:LG} Define two quotient sets of $\mathscr{S}(f,R)^{\times} $using the relations above:
\[
G_R(f):=\mathscr{S}(f,R)^{\times} /\ss, \quad \text{and} \quad G_R^*(R):=\mathscr{S}(f,R)^{\times}/\s.
\] 
\end{df}
We define the products on the sets $G_R(f)$ and $G_R^*(f)$ induced by the abelian group $\mathscr{S}(f,R)^{\times} $.
We can see that the products are well-defined as follows. First, recall that the product in $\mathscr{S}(f,R)^{\times} $
is induced by that of ${\mathcal O}_R^{\times}=(R[t]/(f(t)))^{\times}$ via the isomorphisms (\ref{eq:GS1}).
The products on the sets  $G_R(f)$ and $G_R^*(f)$ are well-defined because the multiplication by $\lambda \in R^{\times}$ on 
$V_f(R)^{\times}$ is equivalent to that on ${\mathcal O}_R^{\times}$, and the action of ${\mathcal B}^{\nu} \ (\nu \in \Z)$
on $V_f(R)^{\times}$ is interpreted as the multiplication of $(P-t)^{\nu}$ on ${\mathcal O}_R^{\times}$ from (\ref{eq:LR4}).
The
identity elements of $G_R(f) $ and $G_R^*(f)$ are the class of the  Lucas sequence ${\mathcal F}$,  the inverse element
of the class  $\displaystyle{[{\bf w}],\  {\bf w}=(w_n), \ w_n=\frac{A\theta_1^n-B\theta_2^n}{\theta_1-\theta_2} }$
is given by $\displaystyle{[{\bf w}]^{-1}=[ {\bf u} ], \ {\bf u} =(u_n),\ u_n=\frac{B\theta_1^n-A\theta_2^n}{\theta_1-\theta_2}}$.

\begin{rem}\label{rem:LG}
There exists a natural bijection between $G_{\Q}^*(f)$ and the Laxton group $G^*(f)$ in \S \ref{sec:LG}.
\end{rem}

We denote by $\left[ \begin{array}{c}
a_1\\
a_0\\
\end{array} \right] $ the class of $ V_f(R)^{\times} /\ss$ containing
$\begin{pmatrix}
a_1\\
a_0\\
\end{pmatrix}$.
For $\left[ \begin{array}{c}
0\\
1\\
\end{array} \right] \in V_f(R)^{\times} /\ss$, we have $
\left[ \begin{array}{c}
0\\
1\\
\end{array} \right]^{-1} =\left[ \begin{array}{c}
P\\
1\\
\end{array} \right]$ from (\ref{eq:LR10}) and hence
(\ref{eq:LR6}) yields 
the action of ${\mathcal B}$ on $V_f(R)^{\times}/\ss :$ 
\begin{equation}\label{eq:EC1}
{\mathcal B}^n \left[ \begin{array}{c}
a_1\\
a_0\\
\end{array} \right]=\left[ 
\begin{array}{c}
0\\
1\\
\end{array} \right]^{-n} * \left[
\begin{array}{c}
a_1\\
a_0\\
\end{array} \right], 
\end{equation}
for any $n\in \Z$ and $\left[ \begin{array}{c}
a_1\\
a_0\\
\end{array} \right] \in V_f(R)^{\times}/ \ss$.
%
%

We identify the group $G_R(f)$ with $V_f(R)^{\times}/\ss$ and the group $G_R^*(f)$ with
$V_f(R)^{\times}/\s$ via the group isomorphisms induced by (\ref{eq:LR1}):
\begin{align}
& G_R(f) \tt V_f(R)^{\times}/\ss, \quad [(w_n)] \mapsto \left[
\begin{array}{c}
w_1\\
w_0\\
\end{array} \right], \label{eq:EC2}\\
& G_R^*(f) \tt V_f(R)^{\times}/\s, \quad [(w_n)] \mapsto \left[
\begin{array}{c}
w_1\\
w_0\\
\end{array} \right]. \label{eq:EC2} \notag
\end{align}
Therefore, we denote a class of $G_R(f)$ or $G_R^*(f)$ by $\left[
\begin{array}{c}
w_1\\
w_0\\
\end{array} \right]$ instead of $[ (w_n) ]$. Consider a natural surjection
\begin{equation}\label{eq:EC3}
\pi :\ G_R(f) \longrightarrow G_R^*(f), \quad \left[
\begin{array}{c}
w_1\\
w_0\\
\end{array} \right] \mapsto \left[
\begin{array}{c}
w_1\\
w_0\\
\end{array} \right]
\end{equation}
and the image of a subgroup $N$ of $G_R(f)$.
\begin{lem}\label{lem:ImPi}
Let $N$ be a subgroup of $G_R(f)$ and $\pi : G_R(f) \to G_R^*(f)$ be the natural surjection.
Then we have a group isomorphism:
\[
N/N\cap {\tiny \left\langle 
\left[
\begin{array}{c}
0\\
1\\
\end{array} \right]
\right\rangle } \tt \pi (N), \quad 
\left[
\begin{array}{c}
w_1\\
w_0\\
\end{array} \right] \bmod{ N\cap {\tiny \left\langle 
\left[
\begin{array}{c}
0\\
1\\
\end{array} \right] \right\rangle} } \mapsto \left[\begin{array}{c}
w_1\\
w_0\\
\end{array} \right].
\]
\end{lem}

\begin{proof}
The  kernel of the restriction map 
$ \pi \mid_N:\ N \to \pi (N)$ is the subgroup
$N \cap \left\{ {\mathcal B}^{\nu} \left[
\begin{array}{c}
1\\
0\\
\end{array} \right] \ \vline \nu \in \Z \right\}$. Since
$ {\mathcal B}^{\nu} \left[
\begin{array}{c}
1\\
0\\
\end{array} \right] =\left[
\begin{array}{c}
0\\
1\\
\end{array} \right]^{-\nu}$ from (\ref{eq:EC1}), we get the group isomorphism:
\[
N/N\cap {\tiny \left\langle 
\left[
\begin{array}{c}
0\\
1\\
\end{array} \right]
\right\rangle } \simeq \pi (N).
\]
\qed
\end{proof}

Note that $\pi (N)$ is the quotient of $N$ by the action of $\langle {\mathcal B} \rangle$.
In particular, we get
\begin{equation}\label{eq:EC4}
G_R(f)/{\tiny  \left\langle 
\left[
\begin{array}{c}
0\\
1\\
\end{array} \right] \right\rangle }
\overset{\sim}{ \underset{\pi}{\longrightarrow}} G_R^*(f), \quad
\left[
\begin{array}{c}
w_1\\
w_0\\
\end{array} \right] \bmod{ {\tiny \left\langle 
\left[
\begin{array}{c}
0\\
1\\
\end{array} \right] \right\rangle } } \mapsto \left[\begin{array}{c}
w_1\\
w_0\\
\end{array} \right].
\end{equation}

Assume that $f$ is irreducible over $R$. We get the following group isomorphism from Theorem~\ref{theo:Vf} (\ref{enum:Vf1}).
\begin{equation}\label{eq:EC5}
\Psi_R:\ G_R(f) \tt R[\theta_1]^{\times}/R^{\times}, \quad \left[ \begin{array}{c}
w_1\\
w_0\\
\end{array} \right] \mapsto w_1-w_0\theta_1 \bmod{R^{\times}}.
\end{equation}
Since $\Psi_R \left( \left[ 
\begin{array}{c}
0\\
1\\
\end{array} \right] \right)=\theta_1 \bmod{R^{\times}}$, the isomorphism induces the following isomorphism
\begin{equation}\label{eq:EC6}
\Psi_R^*:\ G^*_R(f) \tt R[\theta_1]^{\times}/R^{\times} \langle \theta_1 \rangle,  \quad \left[ \begin{array}{c}
w_1\\
w_0\\
\end{array} \right] \mapsto w_1-w_0\theta_1 \bmod{R^{\times} \langle \theta_1 \rangle }.
\end{equation}

Assume that $f$ is reducible over $R$ and $f$ has no double root in $R$. Furthermore, we assume that $R$ is a local ring.
Let $H_R^{\times} =\{ (x,y) \in R^{\times} \times R^{\times}  \mid  x\equiv y \pmod{\sqrt{d}} \}$ be the group in 
Theorem~\ref{theo:Vf}(\ref{enum:Vf2}),(\ref{enum:Vf2i}) (if $d\in R^{\times}$, then $H_R^{\times} =R^{\times} \times R^{\times}$).
We get a group isomorphism
\begin{equation}\label{eq:EC7}
H_R^{\times}/I_R \tt \begin{cases}
R^{\times} & (\text{if}\ d\in R^{\times} ),\\
1+\sqrt{d} \ R & (\text{if}\ d\not\in R^{\times} ),
\end{cases} \quad (x,y) \bmod{I_R} \mapsto xy^{-1} 
\end{equation}
where $I_R:= \{ (x,x)  \mid  x\in R^{\times} \}$ and $1+\sqrt{d} \ R:= \{ 1+\sqrt{d} \ z \mid z\in R \}$. Then we get the following group isomorphism from Theorem~\ref{theo:Vf} (\ref{enum:Vf2}),(\ref{enum:Vf2i}) and
(\ref{eq:EC7}).
\begin{equation}\label{eq:EC8}
\Psi_R:\ G_R(f) \tt 
\begin{cases}
R^{\times} & (\text{if}\ d\in R^{\times} ),\\
1+\sqrt{d} \ R & (\text{if}\ d\not\in R^{\times} ),
\end{cases} \quad \left[ 
\begin{array}{c}
w_1\\
w_0\\
\end{array} \right] \mapsto (w_1-w_0\theta_1)(w_1-w_0 \theta_2)^{-1} .
\end{equation}
Since $\Psi_R \left( \left[ \begin{array}{c}
0\\
1\\
\end{array} \right] \right)=\theta_1 \theta_2^{-1}$, the isomorphism induces the following isomorphism
\begin{align}
\Psi_R^*:\ G_R^*(f) \tt 
\begin{cases}
R^{\times}/\langle \theta_1 \theta_2^{-1} \rangle  & (\text{if}\ d\in R^{\times} ),\\
(1+\sqrt{d} \ R) /\langle \theta_1 \theta_2^{-1} \rangle  &(\text{if}\ d\not\in R^{\times} ),
\end{cases} \label{eq:EC9} \\
 \left[ 
\begin{array}{c}
w_1\\
w_0\\
\end{array} \right] \mapsto (w_1-w_0\theta_1)(w_1-w_0 \theta_2)^{-1}  \bmod{\langle \theta_1 \theta_2^{-1} \rangle}. \notag
\end{align}


Assume that $f$ has a double root $\theta$ in $R$.
From Theorem~\ref{theo:Vf} (\ref{enum:Vf2}),(\ref{enum:Vf2ii}) and the following surjective group homomorphism
\[
(R[\varepsilon]/(\varepsilon^2))^{\times}  \longrightarrow R, \quad x+y\varepsilon \bmod{\varepsilon^2} \mapsto x^{-1}y,
\]
we get the group isomorphism
\begin{equation}\label{eq:EC10}
\Psi_R : G_R(f) \tt R, \quad
\left[
\begin{array}{c}
w_1\\
w_0\\
\end{array} \right] \mapsto -w_0 (w_1-w_0\theta)^{-1}.
\end{equation}
The surjectivity follows from the fact that $\left[ \begin{array}{c}
1+y\theta \\
-y \\
\end{array} \right] \in G_R(f)$ maps to $y\in R$.
Since $\Psi_R \left( \left[ \begin{array}{c}
0\\
1\\
\end{array} \right] \right)=\theta^{-1}$, the isomorphism induces the
following isomorphism
\begin{equation}\label{eq:EC11}
\Psi_R^* : G_R^*(f) \tt R/\langle \theta^{-1} \rangle, \quad \left[
\begin{array}{c}
w_1\\
w_0\\
\end{array} \right] \mapsto -w_0 (w_1-w_0\theta)^{-1} \bmod{ \langle \theta^{-1} \rangle}.
\end{equation}
 
 In the case  $R=\Q$, we denote the discriminant $d$ of $f$ by $D$:
 \[
 D:= P^2-4Q \ \in \Q^{\times}.
 \]
 By summarizing the above discussion, we obtain the following theorem.
 
\begin{theo}\label{theo:G} 
Let $p$ be a prime number with $p\nmid Q$.
\begin{enumerate}[$(1)$]
\item If $f(t)$ is irreducible over $\Q$, then we have
\[
G_{\Q}(f)\tt \Q(\theta_1)^{\times}/\Q^{\times}, \quad \left[ \begin{array}{c}
w_1\\
w_0\\
\end{array} \right]
 \mapsto w_1-w_0\theta_1 \bmod{\Q^{\times}}.
 \]
 
If $f(t)$ is reducible over $\Q$, then we have 
\[
G_{\Q}(f)\tt \Q^{\times}, \quad \left[ \begin{array}{c}
w_1\\
w_0\\
\end{array} \right]
 \mapsto (w_1-w_0\theta_1)(w_1-w_0\theta_2)^{-1}.
\]
\label{enum:G1}
\item  If $f(t)$ is irreducible over $\Q$, then we have
\[
G_{\Z_{(p)}}(f) \tt \Z_{(p)}[\theta_1]^{\times}/\Z_{(p)}^{\times}, \quad 
\left[ \begin{array}{c}
w_1\\
w_0\\
\end{array} \right] \mapsto w_1-w_0\theta_1\bmod{\Z_{(p)}^{\times}}.
\]

If $f(t)$ is reducible over $\Q$, then we have 
\[G_{\Z_{(p)}}(f)\tt \begin{cases} \Z_{(p)}^{\times} & \text{if}\ p\nmid D,\\
1+p^{\frac{s}{2}} \Z_{(p)} & \text{if}\ p|D,
\end{cases} \quad \left[ \begin{array}{c}
w_1\\
w_0\\
\end{array} \right]
 \mapsto (w_1-w_0\theta_1)(w_1-w_0\theta_2)^{-1}.
\]
\label{enum:G2}
\item
\begin{enumerate}[$(i)$]
\item If $f(t) \bmod{p}$ is irreducible over $\F_p$, then we have
\[
G_{\F_p}(f)\tt \F_p(\theta_1)^{\times}/\F_p^{\times}, \quad 
\left[ \begin{array}{c}
w_1\\
w_0\\
\end{array} \right] \mapsto w_1-w_0\theta_1\bmod{\F_p^{\times}}.
\]
\label{enum:G3i}
\item Assume that $f(t)$ is reducible over $\F_p$.

If $p\nmid D$, then we have 
\[
G_{\F_p}(f) \tt \F_p^{\times} 
, \quad \left[ \begin{array}{c}
w_1\\
w_0\\
\end{array} \right]  \mapsto (w_1-w_0\theta_1)(w_1-w_0\theta_2)^{-1}.
\]

If $p|D$, then we have 
\[
G_{\F_p}(f) \tt  \F_p, \quad \left[ \begin{array}{c}
w_1\\
w_0\\
\end{array} \right] \mapsto -w_0(w_1-w_0\theta)^{-1}, 
\]
where $\theta$ is the double root of $f(t) \bmod{p}$.
\label{enum:G3ii}
\end{enumerate}
\label{enum:G3}
\end{enumerate}
\end{theo}
We also have the following theorem for the group $G_R^*(f)$ as well.
\begin{theo}\label{theo:G*} 
 Let $p$ be a prime number with $p\nmid Q$.
\begin{enumerate}[$(1)$]
\item If $f(t)$ is irreducible over $\Q$, then we have
\[
G^*_{\Q}(f)\tt \Q(\theta_1)^{\times}/\Q^{\times}\langle \theta_1 \rangle, \quad 
\left[ \begin{array}{c}
w_1\\
w_0\\
\end{array} \right]  \mapsto w_1-w_0\theta_1 \bmod{\Q^{\times}}\langle \theta_1 \rangle.
\]

If $f(t)$ is reducible over $\Q$, then we have 
\[
G_{\Q}^*(f)\tt \Q^{\times} /
\langle \theta_1 \theta_2^{-1} \rangle, \quad
\left[ \begin{array}{c}
w_1\\
w_0\\
\end{array} \right] \mapsto (w_1-w_0\theta_1)(w_1-w_0\theta_2)^{-1} \bmod{\langle \theta_1 \theta_2^{-1} \rangle}.
\]
\label{enum:G*1}
\item  If $f(t)$ is irreducible over $\Q$, then we have
\[
G_{\Z_{(p)}}^*(f)\tt \Z_{(p)}[\theta_1]^{\times}/\Z_{(p)}^{\times} \langle \theta_1 \rangle,\quad
\left[ \begin{array}{c}
w_1\\
w_0\\
\end{array} \right] \mapsto w_1-w_0\theta_1 \bmod{ \Z_{(p)}^{\times} \langle \theta_1 \rangle}.
\]

If $f(t)$ is reducible over $\Q$, then we have 
\begin{align*}
G_{\Z_{(p)}}^*(f) \tt
\begin{cases}
\Z_{(p)}^{\times}/\langle \theta_1\theta_2^{-1} \rangle & \text{if}\ p\nmid D,\\
(1+p^{\frac{s}{2}}\Z_{(p)})/\langle \theta_1 \theta_2^{-1} \rangle & \text{if}\ p|D,
\end{cases} \\
\quad \left[ \begin{array}{c}
w_1\\
w_0\\
\end{array} \right] \mapsto (w_1-w_0\theta_1)(w_1-w_0\theta_2)^{-1} \bmod{ \langle \theta_1 \theta_2^{-1} \rangle}.
\end{align*}
\label{enum:G*2}
\item 
\begin{enumerate}[$(i)$]
\item If $f(t) \bmod{p}$ is irreducible over $\F_p$, then we have
\[
G_{\F_p}^*(f)\tt \F_p(\theta_1)^{\times}/\F_p^{\times} \langle \theta_1 \rangle ,\quad 
\left[ \begin{array}{c}
w_1\\
w_0\\
\end{array} \right] \mapsto w_1-w_0 \theta_1 \bmod{\F_p^{\times} \langle \theta_1 \rangle}.
\]
 \label{enum:G*3i}
\item Assume that $f(t)$ is reducible over $\F_p$. 

 If $p\nmid D$, then we have 
\[
G_{\F_p}^*(f)\tt \F_p^{\times} /
\langle \theta_1 \theta_2^{-1} \rangle,\quad
\left[ \begin{array}{c}
w_1\\
w_0\\
\end{array} \right] \mapsto (w_1-w_0\theta_1)(w_1-w_0\theta_2)^{-1} \bmod{\langle \theta_1 \theta_2^{-1}\rangle}.
\]

If $p|D$, then we have
\[
G^*_{\F_p}(f) \tt  0.
\]
 \label{enum:G*3ii}
\end{enumerate}
\label{enum:G*3}
\end{enumerate}
\end{theo}

\section{Subgroups of $G_{\Q}(f)$ and $G_{\Q}^*(f)$ }\label{sec:SubG}

In this section, we define natural subgroups $G(f,p),\ K(f,p),\ H(f,p)$ 
of $G_{\Q}(f)$, and $G^*(f,p),\ K^*(f,p),\ H^*(f,p)$ of $G_{\Q}^*(f)$. The definitions of
these groups give natural interpretation of Laxton's subgroups $G^*(f,p)_L,\ K^*(f,p)_L$ and
$H^*(f,p)_L$ in \S~\ref{sec:LG}.

Let $p$ be a prime number with $p\nmid Q$. Consider the projective line over $\F_p$:
\[
\P^1(\F_p) :=\F_p^2/\sim =\left\{ \left[
\begin{array}{c}
a_1\\
a_0\\
\end{array} \right] \in \F_p^2 \ \vline \ a_0\in \F_p^{\times} \ \text{or} \ a_1\in \F_p^{\times} \right\},
\]
where $\left[
\begin{array}{c}
a_1\\
a_0\\
\end{array} \right]=\left[
\begin{array}{c}
b_1\\
b_0\\
\end{array} \right]$ if and only if there exists $c\in \F_p^{\times}$ such that
$a_0=cb_0,\ a_1=cb_1$. For any class $\left[
\begin{array}{c}
w_1\\
w_0\\
\end{array} \right] \in G_{\Q}(f)$, we can choose the representative 
$\begin{pmatrix}
w_1\\
w_0\\
\end{pmatrix}$ so that $w_0,w_1 \in \Z$ and $(w_0,w_1)=1$. Therefore, the reduction map
\begin{equation}\label{eq:SubG1}
\mathrm{red}_p:\ G_{\Q}(f) \longrightarrow \P^1(\F_p), \quad 
\left[
\begin{array}{c}
w_1\\
w_0\\
\end{array} \right] \mapsto \left[
\begin{array}{c}
w_1\\
w_0\\
\end{array} \right] 
\end{equation}
is well-defined. The group $G_{\F_p}(f) \ \left(\simeq V_f(\F_p)^{\times}/\underset{ \F_p}{\sim} \right)$
is a subset of $\P^1(\F_p)$.
\begin{df}\label{df:GKH}
Define subsets $G(f,p),\ K(f,p)$ and $H(f,p)$ of $G_{\Q}(f)$ by
\begin{align*}
G(f,p) & := \left\{ 
 \left[
\begin{array}{c}
w_1\\
w_0\\
\end{array} \right]  \in G_{\Q} (f) \ \vline \ \mathrm{red}_p \left(  \left[
\begin{array}{c}
w_1\\
w_0\\
\end{array} \right] \right) = \left[
\begin{array}{c}
1\\
0\\
\end{array} \right] \right\},\\
K(f,p) & := \left\{ 
 \left[
\begin{array}{c}
w_1\\
w_0\\
\end{array} \right]  \in G_{\Q} (f) \ \vline \ \mathrm{red}_p \left(  \left[
\begin{array}{c}
w_1\\
w_0\\
\end{array} \right] \right) \in G_{\F_p}(f)  \right\},\\
H(f,p) & := \left\{ 
 \left[
\begin{array}{c}
w_1\\
w_0\\
\end{array} \right]  \in G_{\Q} (f) \ \vline \ \mathrm{red}_p \left(  \left[
\begin{array}{c}
w_1\\
w_0\\
\end{array} \right]^n \right) = \left[
\begin{array}{c}
1\\
0\\
\end{array} \right]\ \text{ for some} \ n\in \Z \right\}.
\end{align*}
\end{df}
From the definition of $K(f,p)$, we can see that $K(f,p)$ is a subgroup of $G_{\Q}(f)$.
The reduction map (\ref{eq:SubG1}) induces
the group homomorphism
\begin{equation}\label{eq:SubG2}
\mathrm{red}_p:\ K(f,p)  \longrightarrow G_{\F_p}(f), \quad 
\left[
\begin{array}{c}
w_1\\
w_0\\
\end{array} \right] \mapsto \left[
\begin{array}{c}
w_1\\
w_0\\
\end{array} \right].
\end{equation}
The set $G(f,p)$ is a subgroup of $K(f,p)$ because $\left[
\begin{array}{c}
1\\
0\\
\end{array} \right]$ is the identity element of $G_{\F_p}(f)$. The set $H(f,p)$ is a subgroup of $G_{\Q}(f)$ since
$G(f,p)$ is a group and $\left[
\begin{array}{c}
w_1\\
w_0\\
\end{array} \right] \in H(f,p)$ if and only if
$
\left[
\begin{array}{c}
w_1\\
w_0\\
\end{array} \right]^n \in G(f,p)$ for some $n\in \Z$. Furthermore, we have $K(f,p) \subset H(f,p)$ since the image of 
$K(f,p)$ by the reduction map is in the finite group $G_{\F_p}(f)$.
We get the following sequence of subgroups.
\begin{equation}\label{eq:SubG3}
G(f,p) \leq K(f,p) \leq H(f,p)\leq G_{\Q}(f)
\end{equation}
By the definition of the subgroups, we have an exact sequence of groups:
\begin{equation}\label{eq:SubG4}
1 \longrightarrow G(f,p) \longrightarrow K(f,p) \underset{ \mathrm{red}_p}{\longrightarrow}  G_{\F_p}(f) \longrightarrow 1.
\end{equation}
This exact sequence is an analogue of the elliptic curve exact sequence (\cite[VII, Proposition 2.1]{Si}).

Next, we define the corresponding subgroups of $G_{\Q}^*(f)$ by the natural map $\pi :\ G_{\Q} (f) \to G_{\Q}^*(f)$.
\begin{df}\label{df:GKH*}
Define subsets $G^*(f,p),\ K^*(f,p)$ and $H^*(f,p)$ of $G_{\Q}^*(f)$ by
\begin{align*}
G^*(f,p) & :=\pi (G(f,p)),\\
K^*(f,p) & :=\pi (K(f,p)),\\
H^*(f,p) &:=\pi (H(f,p)).
\end{align*}
\end{df}
These subsets are subgroups of $G_{\Q}^*(f)$ since the map $\pi$ is a group homomorphism and we get the following sequence of
subgroups:
\begin{equation}\label{eq:SubG5}
G^*(f,p) \leq K^*(f,p) \leq H^*(f,p)\leq G^*_{\Q}(f).
\end{equation}
Let $p$ be a prime number, and assume $p\nmid Q$. 
For the Lucas sequence ${\mathcal F}=({\mathcal F}_n) \in \mathscr{S}(f,\Q)$,
we have ${\mathcal F}_n \in \Z_{(p)}$ for any $n\in \Z$. 
 Lucas showed that there exists a positive integer $n$ 
satisfying $p|{\mathcal F}_n$ in this case
(\cite[\S 24, 25]{Lu}, \cite[Lemma~2, Theorem~12]{C}, \cite[IV.18, IV.19 and p67]{R}).
\begin{df}\label{df:rank}
Assume  $p\nmid Q$. We denote {\it the rank} of the Lucas sequence ${\mathcal F}=({\mathcal F}_n) \in \mathscr{S}(f,\Q)$ by
$r(p)$. Namely, it is the smallest positive integer $n$ satisfying $p|{\mathcal F}_n$.
\end{df}
We can easily check $r(2)=2$ if $P$ is even, and $r(2)=3$ if $P$ is odd. If $p\ne 2$, then we know that $r(p)$ divides
$p-\left( \dfrac{D}{p} \right) $ from the results of Lucas, where $\left( \dfrac{*}{*} \right)$ is the Legendre symbol.
\begin{lem}\label{lem:orderF}
Assume $p\nmid Q$. Let $r(p)$ be the rank of the Lucas sequence ${\mathcal F}=({\mathcal F}_n)$. Then
$r(p)$ is equal to the order of 
$
\left[ \begin{array}{c}
0\\
1\\
\end{array} \right] \in G_{\F_p}(f)$.
\end{lem}

\begin{proof}
By the definition, $r(p)$ is the smallest positive integer $n$ satisfying
\[
\left[ \begin{array}{c}
1\\
0\\
\end{array} \right] =\left[ \begin{array}{c}
{\mathcal F}_{n+1} \\
{\mathcal F}_n \\
\end{array} \right] ={\mathcal B}^n\left[ \begin{array}{c}
{\mathcal F}_1\\
{\mathcal F}_0\\
\end{array} \right] 
={\mathcal B}^n\left[ \begin{array}{c}
1\\
0\\
\end{array} \right] 
\]
in $G_{\F_p}(f)$.
Furthermore, we get ${\mathcal B}^n \left[ \begin{array}{c}
1\\
0\\
\end{array} \right] =
\left[ \begin{array}{c}
0\\
1\\
\end{array} \right]^{-n}$ from (\ref{eq:EC1}). Therefore, $r(p)$ is equal to the order of $
\left[ \begin{array}{c}
0\\
1\\
\end{array} \right] $ in $G_{\F_p}(f)$.
\qed
\end{proof}

\begin{lem}\label{lem:G(f,p)}
Let $r(p)$ be the rank of the Lucas sequence ${\mathcal F}=({\mathcal F}_n)$. Then we have
\[
G(f,p) \cap  {\tiny \langle \left[
\begin{array}{c}
0\\
1\\
\end{array} \right] \rangle } = {\tiny
\langle \left[
\begin{array}{c}
0\\
1\\
\end{array} \right]^{r(p)} \rangle.}
\]
\end{lem}
\begin{proof}
The assertion follows from Lemma~\ref{lem:orderF}.
\qed
\end{proof}

We see that $\left[ \begin{array}{c}
0\\
1\\
\end{array} \right] \in K(f,p)$ from $\Lambda (0,1) =Q \not\equiv 0 \pmod{p}$,
hence we get the following isomorphisms by Lemmas~\ref{lem:ImPi} and \ref{lem:G(f,p)}.
\begin{align}
G^*(f,p) & \overset{\sim}{ \underset{\pi}{\longleftarrow}}
G(f,p) /{\tiny  \left\langle \left[ \begin{array}{c}
0\\
1\\
\end{array} \right]^{r(p)} \right\rangle },
\\
K^*(f,p)  & \overset{\sim}{ \underset{\pi}{\longleftarrow}}
K(f,p)/ {\tiny  \left\langle \left[ \begin{array}{c}
0\\
1\\
\end{array} \right] \right\rangle }, \notag \\
H^*(f,p)  & \overset{\sim}{ \underset{\pi}{\longleftarrow}}
H(f,p)/ {\tiny  \left\langle \left[ \begin{array}{c}
0\\
1\\
\end{array} \right] \right\rangle }. \notag 
\end{align}
The exact sequence (\ref{eq:SubG4}) yields the exact sequence of groups:
\begin{equation}\label{eq:SubG8}
1 \longrightarrow G^*(f,p) \longrightarrow K^*(f,p) \underset{ \mathrm{red}_p}{\longrightarrow}  G^*_{\F_p}(f) \longrightarrow 1.
\end{equation}

Let $G^*(f)$ be the Laxton group and $G^*(f,p)_L, K^*(f,p)_L, H^*(f,p)_L$ be the subgroups
defined in \S 1.
For the natural isomorphism of groups $\iota: G^*(f) \to G^*_{\Q}(f), \ {\mathcal W} \mapsto {\mathcal W}$, our subgroups $G^*(f,p),\ K^*(f,p)$ and $H^*(f,p)$ in Definition~\ref{df:GKH*} correspond to $G^*(f,p)_L,\ K^*(f,p)_L$ and
$H^*(f,p)_L$, respectively. 

\begin{theo}\label{theo:main2}
Let $G^*(f)$ be the Laxton group and $G^*(f,p)_L,\ K^*(f,p)_L,
H^*(f,p)_L$ be the subgroups defined in \S~\ref{sec:LG}. By the natural group isomorphism 
of groups $\iota: G^*(f) \to G^*_{\Q}(f)$, we have
$\iota (G^*(f,p)_L)=G^*(f,p),\ \iota (K^*(f,p)_L)=K^*(f,p)$ and $\iota (H^*(f,p)_L)=H^*(f,p)$. 
Furthermore, we have the following exact sequence of groups:
\[
1 \longrightarrow G^*(f,p) \longrightarrow K^*(f,p) \underset{ \mathrm{red}_p}{\longrightarrow}  G^*_{\F_p}(f) \longrightarrow 1,
\]
where the map $\mathrm{red}_p$ is the reduction map defined by {\small $\left[ 
\begin{array}{c}
w_1\\
w_0\\
\end{array} \right] \mapsto \left[ 
\begin{array}{c}
w_1\\
w_0\\
\end{array} \right]$, 
}
 where $\begin{pmatrix}
w_1\\
w_0\\
\end{pmatrix}$ is a representative  such that 
$w_0,w_1 \in \Z$ and $(w_0,w_1)=1$.
\end{theo}

\begin{proof}
We only give the proof of $\iota (G^*(f,p)_L)=G^*(f,p)$, because the other assertions have already been proved.
Let ${\mathcal W} \in G^*(f,p)_L$.
We can choose ${\bf w} =(w_n) \in {\mathcal W}$ such that $w_0, w_1 \in \Z $ and at least one of them is
not divisible by $p$. Since ${\bf w}$ is 
a representative of ${\mathcal W} \in G^*(f,p)_L$, we have $w_n \in p\Z_{(p)}$ for some $n$. For ${\mathcal B}=\begin{pmatrix}
P & -Q\\
1 & 0\\
\end{pmatrix}$, we have
\[
\begin{pmatrix}
w_1\\
w_0\\
\end{pmatrix} ={\mathcal B}^{-n} 
\begin{pmatrix}
w_{n+1} \\
w_n\\
\end{pmatrix},
\]
and hence $w_{n+1} \in \Z_{(p)}^{\times}$ because of $p\nmid w_0$ or $p\nmid w_1$. We conclude that $\iota
({\mathcal W}) \in G^*(f,p)$, and we have
$\iota (G^*(f,p)_L) \subset G^*(f,p)$.
The opposite inclusion relation $\iota (G^*(f,p)_L) \supset G^*(f,p)$ is obvious. 
\qed
\end{proof}

\section{Group strctures of $K(f,p)$ and $K^*(f,p)$}\label{sec:K}
In this section, we determine the group structures of $K(f,p)$ and $K^*(f,p)$ defined in Definitions~\ref{df:GKH} and
\ref{df:GKH*}.
Let $p$ be a prime number with $p\nmid Q$.
\begin{lem}\label{lem:KK*} 
We have the following group isomorphisms.
\begin{enumerate}[$(1)$]
\item
\[
\rho:\ G_{\Z_{(p)}}(f)\tt
 K(f,p), \quad \left[ \begin{array}{c}
 w_1\\
 w_0\\
 \end{array} \right] \mapsto \left[
 \begin{array}{c}
 w_1\\
 w_0\\
 \end{array} \right]
\]
\label{enum:KK*1}
\item
\[
 \rho* :\ G^*_{\Z_{(p)}}(f)
  \tt
 K^*(f,p), \quad 
 \left[ \begin{array}{c}
 w_1\\
 w_0\\
 \end{array} \right]  \bmod{  \langle {\tiny \left[ \begin{array}{c}
0\\
1\\
\end{array} \right] } \rangle   } \mapsto 
  \left[ \begin{array}{c}
 w_1\\
 w_0\\
 \end{array} \right]  \bmod{  \langle {\tiny \left[ \begin{array}{c}
0\\
1\\
\end{array} \right] } \rangle   }.
\]
\label{enum:KK*2}
\end{enumerate}
\end{lem}

\begin{proof}
(\ref{enum:KK*1}) 
We only show that $\rho$ is injective because the other part is trivial.
If $\rho \left( \left[ \begin{array}{c}
w_1\\
w_0\\
\end{array} \right] \right)=\left[ \begin{array}{c}
1\\
0\\
\end{array} \right]$  for $\left[ \begin{array}{c}
w_1\\
w_0\\
\end{array} \right] \in G^*_{\Z_{(p)}}(f)$, then
we get $w_0=0, w_1\in \Q^{\times}$. 
Furthermore, we have $w_1\in \Z_{(p)}^{\times}$ since
$\Lambda (w_1,w_0)=w_1^2-Pw_0w_1+Qw_0^2 \in \Z_{(p)}^{\times}$.
We get $\left[ \begin{array}{c}
w_1\\
w_0\\
\end{array} \right] =\left[ \begin{array}{c}
1\\
0\\
\end{array} \right]$ in $G^*_{\Z_{(p)}}(f)$, and hence $\rho$ is injective.

(\ref{enum:KK*2}) We get the assertion from (\ref{enum:KK*1})  since the kernels of the natural surjection $K(f,p) \to K^*(f,p)$
and $G_{\Z_{(p)} }(f) \to G^*_{\Z_{(p)}} (f)$ are the subgroups generated by $\left[
 \begin{array}{c}
0\\
1\\
\end{array} \right]$.
\qed
\end{proof}

By Lemma~\ref{lem:KK*}, Theorems~\ref{theo:G} (\ref{enum:G2} ) and \ref{theo:G*} (\ref{enum:G*2}), we get the following theorem.
\begin{theo}\label{theo:KK*}
Put $D=p^sD_0$ with $s\geq 0,\ p\nmid D_0$.
\begin{enumerate}[$(1)$]
\item If $f(t)$ is irreducible over $\Q$, then we have
\[
K(f,p) \tt
\Z_{(p)}[\theta_1]^{\times}/\Z_{(p)}^{\times} , \quad \left[
\begin{array}{c}
w_1\\
w_0\\
\end{array} \right] \mapsto w_1-w_0\theta_1 \bmod{\Z_{(p)}^{\times}},
\]
and
\[
K^*(f,p) \tt
\Z_{(p)}[\theta_1]^{\times}/\Z_{(p)}^{\times}\langle \theta_1 \rangle , \left[
\begin{array}{c}
w_1\\
w_0\\
\end{array} \right] \mapsto w_1-w_0\theta_1 \bmod{\Z_{(p)}^{\times}
\langle \theta_1 \rangle}.
\]
\label{enum:theoKK*1}
\item If $f(t)$ is reducible over $\Q$, then we have
\[
K(f,p) \tt \begin{cases}
\Z_{(p)}^{\times} & \text{if}\ p\nmid D,\\
1+p^{\frac{s}{2}} \Z_{(p)} & \text{if}\ p|D,
\end{cases} 
\quad
\left[
\begin{array}{c}
w_1\\
w_0\\
\end{array} \right]  \mapsto (w_1-w_0 \theta_1)(w_1-w_0\theta_2)^{-1},
\]
and
\begin{align*}
K^*(f,p) \tt \begin{cases}
\Z_{(p)}^{\times}/\langle \theta_1 \theta_2^{-1} \rangle  & \text{if}\ p\nmid D,\\
(1+p^{\frac{s}{2}} \Z_{(p)} )/\langle \theta_1 \theta_2^{-1} \rangle & \text{if}\ p|D,
\end{cases} 
\\
\left[
\begin{array}{c}
w_1\\
w_0\\
\end{array} \right] \mapsto (w_1-w_0 \theta_1)(w_1-w_0\theta_2)^{-1} \bmod{\langle \theta_1 \theta_2^{-1} \rangle}.
\end{align*}
\label{enum:theoKK*2}
\end{enumerate}
\end{theo}

\section{ Group structures of $G_{\Q}(f)/K(f,p)$ and $G^*_{\Q}(f)/K^*(f,p)$}\label{sec:G/Kp}

In this section,  we determine  the structures of $G_{\Q}(f)/K(f,p)$ and $G_{\Q}^*(f)/K^*(f,p)$.
Put  $F:=\Q(\theta_1)$ and denote the ring of integers of $F$ by ${\mathcal O}_F$.
Let $p$ be a prime number and ${\mathfrak p}$ be a prime ideal of $F$ above $p$.
Put ${\mathcal O}_{\mathfrak p}:=S_{\mathfrak p}^{-1} {\mathcal O}_F$ and ${\mathcal O}_{(p)}:=
S_p^{-1}{\mathcal O}_F=\bigcap_{ {\mathfrak p} \mid p} {\mathcal O}_{F,{\mathfrak p}}  $,
where $S_{\mathfrak p}:={\mathcal O}_F \setminus {\mathfrak p}$ and $S_p:=\Z\setminus p\Z$.
We have ${\mathcal O}_{(p)}^{\times} =\{ \alpha \in F^{\times} \mid v_{\mathfrak p}(\alpha)=0 \ \text{
for any}\ {\mathfrak p}\ \text{above}\ p \}$.
Put $D=ma^2$ with $a\in \N$ and a squarefree integer $m$. We have
$F=\Q(\sqrt{m})$, and if $p\ne 2$, then ${\mathcal O}_{(p)}=\Z_{(p)} [\sqrt{m} ],\ \Z_{(p)}[\theta_1 ]=\Z_{(p)} [\sqrt{D} ]$.

\begin{lem}\label{lem:localring}
 If $p^2 \nmid D$, then we have ${\mathcal O}_{(p)}=\Z_{(p)} [\theta_1]$.
\end{lem}
\begin{proof}
If $F=\Q$, then the assertion is trivial. Assume $F\ne \Q$.
In the case $p\ne 2$, the assertion follows from 
\[
{\mathcal O}_{(p)}=\Z_{(p)}[\sqrt{m}]=\Z_{(p)} [\sqrt{D} ]=\Z_{(p)}[\theta_1 ].
\]
Assume $p=2$. Since $2^2\nmid D=P^2-4Q$, we have
$2\nmid P$. Furthermore, since $\theta_1=(P\pm \sqrt{D})/2=(P\pm a\sqrt{m})/2 \in {\mathcal O}_F$,
we have
$m\equiv 1\pmod{4}$ and $2\nmid a$, and hence ${\mathcal O}_F=\Z[(1+\sqrt{m})/2 ]$.
We get $(1+\sqrt{m})/2 \in \Z_{(p)} [\theta_1]$ from 
\[
\frac{1\pm \sqrt{m}}{2}=\theta_1-\frac{(P-1)\pm (a-1)\sqrt{m}}{2}.
\]
Therefore, we have ${\mathcal O}_{(p)}=\Z_{(p)}[\theta_1]$.
\qed
\end{proof}
\begin{lem}\label{lem:F/ZpQ}
If $f(t)$ is irreducible over $\Q$, then we have
\[
F^{\times}/\Z_{(p)} [\theta_1]^{\times} \Q^{\times} \simeq
\begin{cases}
{\mathcal O}_{(p)}^{\times}/\Z_{(p)}[\theta_1]^{\times} & \text{if $p$ is inert in $F$},\\
\Z \times ( {\mathcal O}_{(p)}^{\times}/\Z_{(p)}[\theta_1]^{\times} )& \text{if $p$ splits in $F$},\\
\Z/2\Z \times ( {\mathcal O}_{(p)}^{\times}/\Z_{(p)}[\theta_1]^{\times}) & \text{if $p$ is ramified in $F$}.
\end{cases}
\]
\end{lem}
\begin{proof}
If $p$ is inert in $F$, then the assertion follows from
\[
F^{\times}/\Z_{(p)}[\theta_1]^{\times} \Q^{\times} ={\mathcal O}_{(p)}^{\times} \langle p\rangle /\Z_{(p)}[\theta_1]^{\times}
\langle p\rangle \simeq {\mathcal O}_{(p)}^{\times} /\Z_{(p)}[\theta_1]^{\times}.
\]
Next, we consider the case $p$ splits in $F$.
Put Gal$(F/\Q)=\langle \sigma \rangle $ and $(p)={\mathfrak p}{\mathfrak p}^{\sigma}$.
Consider a split surjection
\[
{\mathcal V}: F^{\times}/{\Z_{(p)}[\theta_1]}^{\times} \Q^{\times} \to \Z,\quad {\mathcal V}(\alpha)=v_{\mathfrak p}(\alpha^{1-\sigma})
=v_{\mathfrak p}(\alpha)-v_{{\mathfrak p}^{\sigma}}(\alpha).
\]
We have $\alpha \in$Ker${\mathcal  V}$ if and only if $\alpha=p^n \beta$ for some $n\in \Z$ and $\beta\in 
{\mathcal O}_{(p)}^{\times}$.
Therfore, we get
\[
\mathrm{Ker}{\mathcal V}={\mathcal O}_{(p)}^{\times} \langle  p\rangle /\Z_{(p)}[\theta_1]^{\times} \Q^{\times}
={\mathcal O}_{(p)}^{\times} \langle p\rangle /\Z_{(p)} [\theta_1]^{\times} \langle p\rangle \simeq
{\mathcal O}_{(p)}^{\times} /\Z_{(p)}[\theta_1]^{\times}, 
\]
and hence 
\[
F^{\times}/\Z_{(p)}[\theta_1]^{\times} \Q^{\times} \simeq \Z\times ({\mathcal O}_{(p)}^{\times}/\Z_{(p)}[\theta_1]^{\times}).
\]
Finally, we consider the case $p$ is ramified in $F$.
Put $(p)={\mathfrak p}^2$, and choose $\pi \in {\mathfrak p}\setminus {\mathfrak p}^2$.
The assertion follows from the isomorphism
\[
{\mathcal W}: F^{\times} \overset{\sim}{\longrightarrow} {\mathcal O}_{(p)}^{\times}
\times \Z,\quad {\mathcal W}(\alpha)=(\pi^{-v_{\mathfrak p}(\alpha)} \alpha, v_{\mathfrak p}(\alpha)),
\]
and 
\[
{\mathcal W} (\Z_{(p)}[\theta_1]^{\times} \Q^{\times})
={\mathcal W}(\Z_{(p)}[\theta_1]^{\times} \langle p\rangle )
= \Z_{(p)} [\theta_1]^{\times} \times 2\Z.
\]
\qed
\end{proof}
\begin{lem}\label{lem:O/Zp}
Assume that $f(t)$ is irreducible over $\Q$.
Put $D=p^s D_0$ with $s\geq 0, \ p\nmid D_0$.
\begin{enumerate}[$(1)$]
\item If $p$ is inert in $F$, then we have
\[
{\mathcal O}_{(p)}^{\times}/\Z_{(p)}[\theta_1]^{\times}  \simeq \begin{cases}
0 & \text{if $s=0$},\\
\Z/p^{s/2-1}(p+1)\Z & \text{if  $s\ne0, \ s\equiv 0\pmod{2}$ and $p\ne 2$}.
\end{cases}
\]
\label{enum:O/Zp1}
\item If $p$ splits in $F$, then we have
\[
{\mathcal O}_{(p)}^{\times} /\Z_{(p)}[\theta_1]^{\times}\simeq \begin{cases}
0 & \text{if $s=0$},\\
\Z/p^{s/2-1}(p-1)\Z & \text{if  $s\ne0, \ s\equiv 0\pmod{2}$ and $p\ne 2$}.
\end{cases}
\]
\label{enum:O/Zp2}
\item If $p$ is ramified in $F$, then we have
\begin{align*}
& {\mathcal O}_{(p)}^{\times} /\Z_{(p)}[\theta_1]^{\times} \\
 &  \simeq \begin{cases}
0 & \text{if $s=1$},\\
\Z/p^{[s/2]}\Z & \text{if  $s\ne1, \ s\equiv 1\pmod{2}$ and $p\ne 2,3$},
\\
\Z/p^{[s/2]}\Z  \ \text{or}  \
\Z/p\Z \times \Z/p^{[s/2]-1}\Z
&
 \text{if  $s\ne1, \ s\equiv 1\pmod{2}$ and $p=3$.}
\end{cases}
\end{align*}
\label{enum:O/Zp3}
\end{enumerate}
\end{lem}
\begin{proof}
The assertions in the cases $s=0,1$ follow from Lemma~\ref{lem:localring}.
Consider the cases $s\ne 0,1$ and $p\ne 2$.
We have ${\mathcal O}_{(p)}=\Z_{(p)}[\sqrt{m}],\ \Z_{(p)}[\theta_1]=\Z_{(p)}[\sqrt{D}]$ since $p\ne 2$.
Put $k=[s/2] (\geq 1)$. From the following commutative diagram:
\begin{equation*}
\begin{CD}
1 @>>> 1+p^k {\mathcal O}_{(p)} @>>> {\mathcal O}_{(p)}^{\times}  @>>> ({\mathcal O}_F/(p^k))^{\times} @>>>1\\
   @. @|   @AAA  @AAA \\
   1 @>>> 1+p^k \Z_{(p)}[\sqrt{m}] 
 @>>> \Z_{(p)} [\theta_1]^{\times} @>>> (\Z/p^k\Z)^{\times}      @>>>1, \\
\end{CD}
\end{equation*} 
where the middle and right vertical maps are injective,
we get
\begin{equation}\label{eq:O/Zp1}
{\mathcal O}_{(p)}^{\times} /\Z_{(p)}[\theta_1]^{\times} \simeq 
({\mathcal O}_F/(p^k))^{\times}/(\Z/p^k\Z)^{\times}.
\end{equation}

(\ref{enum:O/Zp1}) The assertion in the case where $p$ is inert  in $F$ follows from the
isomorphisms:
$$
\begin{array}{ccccc}
& & & (id, \frac{1}{p} \mathrm{log}_{(p)}) & \\
({\mathcal O}_F/(p^k))^{\times} & \simeq & ({\mathcal O}_F/(p))^{\times}
\times
(1+(p))/(1+(p^k)) & \overset{\sim}{\longrightarrow} &
({\mathcal O}_F/(p))^{\times} \times {\mathcal O}_F/(p^{k-1})\\
& & & &\\
\cup \mid & & \cup \mid & & \cup \mid\\
& & & (id, \frac{1}{p} \mathrm{log}_p) & \\
(\Z/p^k\Z)^{\times} &\simeq &(\Z/p\Z)^{\times} \times (1+p\Z)/(1+p^k\Z) & 
\overset{\sim}{\longrightarrow} & (\Z/p\Z)^{\times} \times \Z/p^{k-1}\Z,
\end{array}
$$
and (\ref{eq:O/Zp1}).

(\ref{enum:O/Zp2}) Consider the case where $p$ splits  in $F$.  
Put Gal$(F/\Q)=\langle \sigma \rangle$ and $(p)={\mathfrak p}
{\mathfrak p}^{\sigma}$.
Then we have
$$
\begin{array}{rcl}
({\mathcal O}_F/(p^k))^{\times} & \simeq & ({\mathcal O}_F/(p))^{\times}
\times (1+{\mathfrak p})/(1+{\mathfrak p}^k) \times (1+{\mathfrak p}^{\sigma})
/(1+({\mathfrak p}^{\sigma})^k) \\
& &\\
& (id, \mathrm{log}_{\mathfrak p}, \mathrm{log}_{{\mathfrak p}^{\sigma}}) & \\
& \overset{\sim}{\longrightarrow} & ({\mathcal O}_F/(p))^{\times} \times
{\mathfrak p}/{\mathfrak p}^k \times {\mathfrak p}^{\sigma}/({\mathfrak p}^{\sigma}
)^k \\
& &\\
& \simeq & ({\mathcal O}_F/(p))^{\times} \times (p)/(p^k)\\
& &\\
& (id,\frac{1}{p} ) &\\
& \overset{\sim}{\longrightarrow} & ({\mathcal O}_F/(p))^{\times}
\times {\mathcal O}_F/(p^{k-1}),
\end{array}
$$ 
and hence
$$
\begin{array}{ccc}
({\mathcal O}_F/(p^k))^{\times}  & \simeq & ({\mathcal O}_F/(p))^{\times}
\times {\mathcal O}_F/(p^{k-1}) \\
& &\\
\cup \mid & & \cup \mid \\
& &\\
(\Z/p^k\Z)^{\times} & \simeq & (\Z/p\Z)^{\times} \times \Z/p^{k-1}\Z.
\end{array}
$$
The assertion follows from these  isomorphisms and (\ref{eq:O/Zp1}).

((\ref{enum:O/Zp3})) Consider the case where $p$ is ramified.
Put $(p)={\mathfrak p}^2$,
\[
\ell:=\begin{cases}
1 & \text{if}\ p\ne 2,3,\\
2 & \text{if}\ p=3,
\end{cases}
\]
ane choose $\pi \in {\mathfrak p}^{\ell} \setminus {\mathfrak p}^{\ell+1}$.
We have a commutative diagram:
\begin{equation}\label{eq:O/Zp2}
\begin{CD}
1 @>>> (1+{\mathfrak p}^{\ell})/(1+{\mathfrak p}^{2k})
 @>>> ({\mathcal O}_F/{\mathfrak p}^{2k})^{\times}  @>>> ({\mathcal O}_F/{\mathfrak p}^{\ell})^{\times} @>>>1\\
   @. @AAA   @AAA  @AAA \\
   1 @>>> (1+p \Z)/(1+p^k\Z)
 @>>> (\Z/p^k\Z)^{\times} @>>> (\Z/p\Z)^{\times}      @>>>1, \\
\end{CD}
\end{equation} 
where all the vertical maps are injective, and
\[
({\mathcal O}_F/{\mathfrak p}^{\ell})^{\times} /(\Z/p\Z)^{\times}
\simeq \begin{cases}
0 & \text{if}\ p\ne 2,3,
\\
\Z/p\Z & \text{if}\ p=3.
\end{cases}
\]
Define the injection $\iota :\Z/p^{k-1}\Z \to
{\mathcal O}_F/{\mathfrak p}^{2k-\ell}$
by 
\[
\iota (\alpha)=\begin{cases}
\pi \alpha & \text{if}\ p\ne 2,3,\\
\alpha & \text{if}\ p=3.
\end{cases}
\]
We have 
\[
\text{Coker} \ \iota \simeq \begin{cases}
\Z/p^k\Z & \text{if}\ p\ne 2,3,\\
\Z/p^{k-1}\Z & \text{if}\ p=3.
\end{cases}
\]
The assertion follows from (\ref{eq:O/Zp1}), (\ref{eq:O/Zp2})
and the following commutative diagram:
$$
\begin{array}{ccccc}
& \mathrm{log}_{\mathfrak p} & & \pi & \\
(1+{\mathfrak p}^{\ell})/
(1+{\mathfrak p}^{2k})  & \simeq & {\mathfrak p}^{\ell}/{\mathfrak p}^{2k}
& \overset{\sim}{\longleftarrow} &
 {\mathcal O}_F/{\mathfrak p}^{2k-\ell}\\
& & & &\\
\uparrow  & & \uparrow & &  \iota \uparrow \\
& \mathrm{log}_p & & p & \\
(1+p\Z)/(1+p^k\Z) &\longrightarrow  & p\Z/p^k\Z & 
\longleftarrow & \Z/p^{k-1}\Z,
\end{array}
$$
where all the vertical maps are injective.
\qed
\end{proof}
\begin{cor}\label{cor:G/K}
Put $D=p^sD_0$ with $s\geq 0,\ p\nmid D_0$.
\begin{enumerate}[$(1)$]
\item Assume that $f(t)$ is irreducible over $\Q$.
We have
\[
G_{\Q}(f)/K(f,p) \simeq G_{\Q}^*(f)/K^*(f,p) \simeq F^{\times}
/\Z_{(p)}[\theta_1]^{\times} \Q^{\times}.
\]
\begin{enumerate}[$(i)$]
\item If $p$ is inert in $F$, then
\[
G_{\Q}(f)/K(f,p) 
\simeq \begin{cases}
0 & \text{if}\ s=0,\\
\Z/p^{s/2-1}(p+1)\Z & \text{if}\ s\ne 0,s\equiv 0\pmod{2} \
\text{and}\ p\ne 2.
\end{cases}
\]
\label{enum:G/K1i}
\item If $p$ splits  in $F$, then
\[
G_{\Q}(f)/K(f,p) 
\simeq \begin{cases}
\Z & \text{if}\ s=0,\\
\Z \times \Z/p^{s/2-1}(p-1)\Z & \text{if}\ s\ne 0,s\equiv 0\pmod{2} \
\text{and}\ p\ne 2.
\end{cases}
\]
\label{enum:G/K1ii}
\item If $p$ is ramified in $F$, then
\begin{align*}
& G_{\Q}(f)/K(f,p) 
\simeq  \\ 
& \qquad \begin{cases}
\Z/2\Z & \text{if}\ s=1,\\
\Z/2p^{[s/2]}\Z & \text{if}\ s\ne 1,s\equiv 1\pmod{2} \
\text{and}\ p\ne 2,3,\\
\Z/2p^{[s/2]}\Z \ \text{or}\ \Z/2p\Z \times \Z/p^{[s/2]-1}\Z
& \text{if}\ s\ne 1,s\equiv 1\pmod{2} \
\text{and}\ p\ne 3.\\
\end{cases}
\end{align*}
\label{enum:G/K1iii}
\end{enumerate}
\label{enum:G/K1}
\item Assume that $f(t)$ is reducible over $\Q$ $($hence $s$ is even$)$.
We have
\begin{align*}
G_{\Q}(f)/K(f,p) & \simeq G_{\Q}^*(f)/K^*(f,p) \\
& \simeq
 \begin{cases}
\Q^{\times}/\Z_{(p)}^{\times}\simeq \Z &  \text{if}\ s=0,\\
\Q^{\times}/(1+p^{s/2}\Z_{(p)}) \simeq 
\Z \times \Z/(p-1)p^{s/2-1} \Z & \text{if}\ s\ne 0\
\text{and}\ p\ne 2.
\end{cases}
\end{align*}
\label{enum:G/K2}
\end{enumerate}
\end{cor}
\begin{proof}
(\ref{enum:G/K1}) The first assertion follows from Theorems~\ref{theo:G},
\ref{theo:G*} and \ref{theo:KK*},
and the others follows from Lemmas~\ref{lem:F/ZpQ} and  \ref{lem:O/Zp}.

(\ref{enum:G/K2}) We get from Theorems~\ref{theo:G},
\ref{theo:G*} and \ref{theo:KK*},
\begin{align*}
G_{\Q}(f)/K(f,p) & \simeq G_{\Q}^*(f)/K^*(f,p) \\
& \simeq
 \begin{cases}
\Q^{\times}/\Z_{(p)}^{\times} \simeq \Z &  \text{if}\ p\nmid D,\\
\Q^{\times}/(1+p^{s/2}\Z_{(p)}) & \text{if}\  p\mid D.\
\end{cases}
\end{align*}
Furthermore, if $p\ne 2$, then we have
\[
\Q^{\times} \simeq \Z \times \Z_{(p)}^{\times} \simeq \Z \times
(\Z/p\Z)^{\times}
\times (1+p\Z_{(p)}),
\]
and
$$
\begin{array}{ccccc}
(1+p\Z_{(p)})/(1+p^{s/2}\Z_{(p)}) & \simeq & (1+p\Z)/(1+p^{s/2}\Z)
& \overset{\sim}{\longrightarrow} & \Z/p^{s/2-1}\Z,\\
& & & \frac{1}{p} \mathrm{log}_p & 
\end{array}
$$
and hence we get
\[
\Q^{\times}/(1+p^{s/2}\Z_{(p)})  \simeq 
\Z \times \Z/(p-1)p^{s/2-1}\Z.
\]
\qed
\end{proof}
\section{Group structures of $K(f,p)/G(f,p)$ and $K^*(f,p)/G^*(f,p)$}\label{sec:Kp/Gp}

In this section, we determine the structure of the quotient groups $K(f,p)/G(f,p)$ and $K^*(f,p)/G^*(f,p)$,
where $K(f,p)$ and $G(f,p)$ are defined in Definition~\ref{df:GKH}, and $K^*(f,p)$ and $G^*(f,p)$ are
defined in Definition~\ref{df:GKH*}.
By the results of this section and \S~\ref{sec:G/Kp}, we get Laxton's theorem (Theorem~\ref{theo:Laxton} in \S~\ref{sec:LG}),
and a result proved by Suwa in the case $p^2|D$.
Let $p$ be a prime number with $p\nmid Q$.
From exact sequences (\ref{eq:SubG4}) and (\ref{eq:SubG8}), we get group isomorphisms
\begin{equation}\label{eq:Kp/Gp1}
K(f,p)/G(f,p) \simeq G_{\F_p}(f), \quad K^*(f,p)/G^*(f,p) \simeq G^*_{\F_p} (f).
\end{equation}
By Lemma~\ref{lem:ImPi}, we have

\begin{equation}\label{eq:Kp/Gp2}
G_{\F_p}^*(f) \simeq G_{\F_p}(f)/ {\tiny  \left\langle \left[
\begin{array}{c}
0\\
1\\
\end{array} \right] \right\rangle }.
\end{equation}
From (\ref{eq:Kp/Gp1}), (\ref{eq:Kp/Gp2}), Theorems~\ref{theo:G}, \ref{theo:G*} and Lemma~\ref{lem:orderF}, we get the following theorem.

\begin{theo}\label{theo:Gp}
\begin{enumerate}[$(1)$]
\item Assume that $f(t) \bmod{p}$ is irreducible over $\F_p$. We have
\begin{align*}
& K(f,p)/G(f,p) \simeq \F_p(\theta_1)^{\times}/\F_p^{\times} \simeq \Z/(p+1)\Z, \\
& \left[ \begin{array}{c}
w_1\\
w_0\\
\end{array} \right] \bmod{G(f,p)} 
\mapsto w_1-w_0\theta_1 \bmod{\F_p^{\times}},
\end{align*}
and 
\begin{align*}
& K^*(f,p)/G^*(f,p) \simeq \F_p(\theta_1)^{\times}/\F_p^{\times}\langle \theta_1 \rangle\simeq \Z/ ((p+1)/r(p)) \Z, \\
& \left[ \begin{array}{c}
w_1\\
w_0\\
\end{array} \right] \bmod{G^*(f,p)} 
\mapsto w_1-w_0\theta_1 \bmod{\F_p^{\times} \langle \theta_1 \rangle},
\end{align*}
\label{enum:Gp1}
\item Assume that $f(t) \bmod{p}$ is reducible over $\F_p$. 

If $p\nmid D$, then we have
\begin{align*}
& K(f,p)/G(f,p) \simeq \F_p^{\times}\simeq \Z/(p-1)\Z, \\
&  \left[ \begin{array}{c}
w_1\\
w_0\\
\end{array} \right] \bmod{G(f,p)} 
\mapsto (w_1-w_0\theta_1)(w_1-w_0\theta_2)^{-1},
\end{align*}
and 
\begin{align*}
& K^*(f,p)/G^*(f,p) \simeq \F_p^{\times}/\langle \theta_1 \theta_2^{-1} \rangle\simeq \Z/ ((p-1)/r(p)) \Z, \\
& \left[ \begin{array}{c}
w_1\\
w_0\\
\end{array} \right]  \bmod{G^*(f,p)} 
\mapsto ( w_1-w_0\theta_1)(w_1-w_0\theta_2)^{-1}  \bmod{\langle \theta_1\theta_0^{-1} \rangle}.
\end{align*}

If $p|D$, then we have
\[ K(f,p)/G(f,p) \simeq \F_p, \quad \left[ \begin{array}{c}
w_1\\
w_0\\
\end{array}  \right] \bmod{G(f,p)} \mapsto -w_0 (w_1-w_0 \theta)^{-1},
\]
where  $\theta$ is the double root of $f(t) \bmod{p}$,
and
\[
K^*(f,p)/G^*(f,p)\simeq 0.
\]
\label{enum:Gp2}
\end{enumerate}
\end{theo}

Since the groups in (\ref{eq:Kp/Gp1}) are finite groups, we can see
\begin{align}
\mathrm{Tor}_{\Z}(G_{\Q}(f)/K(f,p))=H(f,p)/K(f,p), \label{eq:Kp/Gp3} \\
 \mathrm{Tor}_{\Z}(G^*_{\Q}(f)/K^*(f,p))=H^*(f,p)/K^*(f,p). \notag
\end{align}

By (\ref{eq:Kp/Gp3}), Corollary~\ref{cor:G/K} and Theorem~\ref{theo:Gp}, we see that our results lead Laxton's
(Theorem~\ref{theo:Laxton})  and Suwa's theorems.

\begin{cor}\label{cor:main}
 Let $p$ be a prime number with $p\nmid Q$, and $r(p)$ be the rank of the Lucas sequence ${\mathcal F}$.
\begin{enumerate}[$(1)$]
\item Assume $p\nmid D$. If $\mathbb Q(\theta_1)\ne \mathbb Q$ and $p$ is inert in $\mathbb Q(\theta_1)$,
 then $G_{\Q}^*(f)=H^*(f,p)=K^*(f,p)$ and
$G^*(f)/G^*(f,p)$ is a cyclic group of order $(p+1)/r(p)$.
\label{enum:cormain1}
\item Assume $p\nmid D$. If $\Q (\theta_1)=\Q$,  or  $\Q(\theta_1)\ne \Q$ and
$p$ splits in $\Q(\theta_1)$, then $G_{\Q}^*(f)/H^*(f,p)$ is an infinite cyclic group,
and $H^*(f,p)=K^*(f,p)$ and $H^*(f,p)/G^*(f,p)$ is a cyclic group of order $(p-1)/r(p)$.
\label{enum:cormain2}
\item If  $p|D$ and $p^2\nmid D$, then $G_{\Q}^*(f)=H^*(f,p)$ and $K^*(f,p)=G^*(f,p)$.
Furthermore, if $p\ne 2$, then $G_{\Q}^*(f)/G^*(f,p)$ is a cyclic group of order two.
\label{enum:cormain3}
\item
Assume $D=p^sD_0$ with $s\geq 2$ and $p\nmid D_0$.
\begin{enumerate}[$(i)$]
\item Assume $s\equiv 1 \pmod{2}$.  We have $G_{\Q} ^*(f)=H^*(f,p)$ and $K^*(f,p)=G^*(f,p)$.
Furthermore, if $p\ne 2,3 $, then $G_{\Q}^*(f)/G^*(f,p)$ is a cyclic group of order $2p^{[s/2]}$,
and if $p=3$, then $G_{\Q}^*(f)/G^*(f,p)$ is a cyclic group of order $2p^{[s/2]}$ or a direct product of
two cyclic groups of order $2p$ and $p^{[s/2]-1}$.
\label{eum:cormain4i}
\item  Assume  $s\equiv 0 \pmod{2}$.  If $\Q(\theta_1)\ne \Q$ and $p$ is inert in $\Q(\theta_1)$, then $G_{\Q}^*(f)=H^*(f,p)$ and
$K^*(f,p)=G^*(f,p)$.
Furthermore, if $p\ne 2$, then 
$G_{\Q}^*(f)/G^*(f,p)$ is a cyclic group of  order $(p+1)p^{s/2-1}$.
\label{eum:cormain4ii}
\item  Assume  $s\equiv 0 \pmod{2}$.  If $\Q(\theta_1)=\Q$ or $\Q(\theta_1)\ne \Q$, $p$ splits in $\Q(\theta_1)$, then $K^*(f,p)=G^*(f,p)$.
Furthermore, if $p\ne 2$, then $G_{\Q}^*(f)/H^*(f,p)$ is an infinite cyclic group,
$H^*(f,p)/K^*(f,p)$ is a cyclic group of  order $(p-1)p^{s/2-1}$.
\label{eum:cormain4iii}
\end{enumerate}
\label{enum:cormain4}
\end{enumerate}
\end{cor}


\begin{thebibliography}{}
\bibitem[1]{B} Ballot, C.: Density of prime divisors of linear recurrences, Mem.\, Amer.\,  Math.\,  Soc. 115, no. 551  (1995)
\bibitem[2]{C} Carmichael, R.~D.: On the numerical factors of the arithmetic forms $\alpha^n \pm \beta^n$, Ann.\,  of Math.\, (2) 15, 30--70 (1913,1914)
\bibitem[3]{L1} Laxton, R.~R.: On groups of linear recurrences, I,  Duke Math. J.   36, 721--736 (1969)
\bibitem[4]{L2} Laxton, R.~R.: On groups of linear recurrences, II, Elements of finite order. Pacific J. Math. 32, 173--179 (1970)
\bibitem[5]{Lu} Lucas, E.: Th\'{e}orie des fonctions num\'{e}riques simplement p\'{e}riodiques, Amer.\, J.\, Math. 1, 184--240 
and 289--321 (1878)
\bibitem[6]{R} Ribenboim P.: The new book of prime number records, Springer-Verlag, New York (1996)
\bibitem[7]{K} Koshy, T.: Fibonacci and Lucas Numbers with Applications, Pure and Applied Mathematics (2001)
\bibitem[8]{Si} Silverman, J.~H.: The Arithmetic of Elliptic Curves, GTM 106, Springer-Verlag, New York, (1986)
\bibitem[9]{S} Suwa, N: Geometric aspects of Lucas sequences, I, Chuo University  preprint series, NO.~122,
http://ir.c.chuo-u.ac.jp/repository/search/binary/p/10888/s/9824/ (2018)
%
%
\end{thebibliography}
\end{document}